\documentclass[10pt,reqno]{amsart}

\usepackage{microtype}

\usepackage{amsthm}
\usepackage{amssymb}
\usepackage{mathdots}

\usepackage{tikz}
\usetikzlibrary{decorations.pathreplacing,decorations.markings,calc,patterns}
\usepackage{tikz-cd}
\tikzset{>=latex}

\usepackage{enumerate}

\usepackage{ytableau}
\ytableausetup{centertableaux}

\newcommand{\map}[2]{\,{:}\,#1\!\longrightarrow\!#2}
\newcommand{\setc}[2]{\left\{\:\!#1\:\middle|\;\!#2\!\:\right\}}
\newcommand{\set}[1]{\left\{\,\!#1\,\!\right\}}

\newcommand{\CC}{\mathbb{C}}
\newcommand{\RR}{\mathbb{R}}

\newcommand{\ZZ}{\mathbb{Z}}
\newcommand{\NN}{\mathbb{N}}
\newcommand{\PP}{\mathbb{P}}

\newcommand{\II}{\mathbb{I}}

\newcommand{\Gg}{\mathcal{G}}
\newcommand{\Ss}{\mathcal{S}}

\newcommand{\Oo}{\mathcal{O}}

\newcommand{\Ll}{\mathcal{L}}
\newcommand{\Ff}{\mathcal{F}}
\newcommand{\Ii}{\mathcal{I}}
\newcommand{\Kk}{\mathcal{K}}
\newcommand{\Ee}{\mathcal{E}}

\newcommand{\Mm}{\mathcal{M}}

\newcommand{\ox}{\otimes}
\newcommand{\lamb}{\lambda_\bullet}
\newcommand{\Tb}{T_\bullet}
\newcommand{\bs}{\backslash}
\newcommand{\deq}{\stackrel{\text{\tiny def}}{=}}
\newcommand{\abs}[1]{\left| #1 \right|}
\newcommand{\sq}{\square}
\newcommand{\glr}{\mathfrak{gl}_r}
\newcommand{\Llambmu}{L(\lamb)^{\mathsf{sing}}_\mu}
\newcommand{\Blambmu}{\crys(\lamb)^{\mathsf{sing}}_\mu}
\newcommand{\Gal}{\mathrm{Gal}}

\newcommand{\cc}{\mathsf{c}}

\newcommand{\associahedron}{\Theta}
\newcommand{\bethe}{A}
\newcommand{\crys}{\mathtt{B}}
\newcommand{\bspec}{\mathcal{A}}
\newcommand{\bdiffop}{\mathcal{D}}
\newcommand{\rybalg}{\overline{A}}
\newcommand{\rybspec}{\overline{\mathcal{A}}}

\newcommand{\SYT}{\mathtt{SYT}}
\newcommand{\SSYT}{\mathtt{SSYT}}
\newcommand{\decgd}[1]{\mathtt{decgd}(#1)}
\newcommand{\RSK}{\mathtt{RSK}}
\newcommand{\words}{\mathtt{words}}
\newcommand{\QRSK}{\mathtt{Q}}
\newcommand{\Rect}{\mathtt{Rect}}
\newcommand{\Part}{\mathtt{Part}}
\newcommand{\evac}{\mathtt{evac}}
\newcommand{\ttN}{\mathtt{N}}

\newcommand{\yt}[1]{\:\ytableausetup{smalltableaux,centertableaux}
                      \ytableaushort{#1}
                      \ytableausetup{nosmalltableaux}\:}

\newcommand{\M}[1]{\overline{M}_{0,#1}}
\newcommand{\oM}[1]{M_{0,#1}}

\DeclareMathOperator{\Hom}{Hom}
\DeclareMathOperator{\End}{End}
\DeclareMathOperator{\Spec}{Spec}

\DeclareMathOperator{\Gr}{Gr}

\DeclareMathOperator{\id}{id}

\def\SheafHom{\mathop{\mathcal{H}\!{\it om}}\nolimits}
\DeclareMathOperator{\Aff}{Aff}

\newcommand{\AxisRotator}[1][rotate=0]{%
    \tikz [x=0.25cm,y=0.60cm,line width=.2ex,-stealth,#1,scale=0.75] \draw (0,0) arc (-150:150:1 and 1);%
}

\newtheorem{Theorem}[equation]{Theorem}
\newtheorem{Proposition}[equation]{Proposition}
\newtheorem{Lemma}[equation]{Lemma}
\newtheorem{Corollary}[equation]{Corollary}

\theoremstyle{definition}
\newtheorem{Definition}[equation]{Definition}
\newtheorem{Example}[equation]{Example}
\newtheorem{Remark}[equation]{Remark}

\numberwithin{equation}{section}

\title{The monodromy of real Bethe vectors for the Gaudin model}
\author{Noah White}
\address{Noah White, School of Mathematics, James Clerk Maxwell Building, The King's Buildings, Peter Guthrie Tait Road, Edinburgh, EH9 3FD, UK.}
\email{noah.white@ed.ac.uk}

\usepackage[colorlinks]{hyperref}

\begin{document}

\begin{abstract}
The Bethe algebras for the Gaudin model act on the multiplicity space of tensor products of irreducible \( \glr \)-modules and have simple spectrum over real points. This fact is proved by Mukhin, Tarasov and Varchenko who also develop a relationship to Schubert intersections over real points. We use an extension to \( \M{n+1}(\RR) \) of these Schubert intersections, constructed by Speyer, to calculate the monodromy of the spectrum of the Bethe algebras. We show this monodromy is described by the action of the cactus group \( J_n \) on tensor products of irreducible \( \glr \)-crystals.
\end{abstract}
\thanks{The author was supported by an Edinburgh Global Research scholarship and a stipend from the School of Mathematics, University of Edinburgh.}

\maketitle
\tableofcontents

\section{Introduction}

\subsection{Gaudin Hamiltonians}
\label{sec:gaudin-hamiltonians}

The Hamiltonians for the Gaudin model are \( n \) commuting operators depending on distinct complex parameters \( z_1,z_2,\ldots,z_n \) acting on a tensor product of irreducible representations of \( \glr \). The problem considered in this paper is to describe the \emph{Galois} or \emph{monodromy group} of these operators.

Let \( \glr \) be the Lie algebra of \( r \times r \) matrices and \( e_{ij} \) the matrix with a \( 1 \) in the \( (i,j) \)-entry and \( 0 \) everywhere else. For \( z = (z_1,z_2,\ldots,z_n) \) a set of distinct complex parameters, the \emph{Gaudin Hamiltonians} are
\begin{equation}
  \label{eq:hamiltonians}
  H_a(z) = \sum_{b \neq a} \frac{\Omega_{ab}}{z_a - z_b} \quad{where} \quad \Omega_{ab} = \sum_{i,j} e^{(a)}_{ij}e^{(b)}_{ji},
\end{equation}
for \( a = 1,2,\ldots,n \). We consider these either as elements of \( U(\glr)^{\ox n} \) or as operators which act on \( 
L(\lamb) = L(\lambda_{1})\ox L(\lambda_{2})\ox\cdots\ox L(\lambda_{n}) \) for an \( n \)-tuple of partitions \( \lamb = (\lambda_{1},\lambda_{2},\ldots,\lambda_{n}) \) with at most \( r \) rows. For \( X \in U(\glr) \), 
\begin{equation*}
X^{(a)} = 1\ox\cdots\ox 1\ox X\ox 1\ox \cdots \ox 1,
\end{equation*}
where \( X \) is placed in the \( a^{\text{th}} \) factor.

The main problem in the study of these operators is to produce a complete family of simultaneous eigenvectors (whenever the operators are diagonalisable), see for example~\cite{Reshetikhin:1995vs}. We consider the problem of understanding how these eigenvectors change as we vary the parameter \( z \).

These operators commute with the action of \( \glr \) on \( L(\lamb) \) (see~\cite{Mukhin:2010ky}). In particular this implies the Hamiltonians preserve weight spaces and singular vectors, hence it is enough to understand the action of the \( H_a(z) \) on the space \( \Llambmu \) for some partition \( \mu \). Let \( G(\lamb;z)_\mu \) be the commutative subalgebra of \( \End(\Llambmu) \) generated by the operators~(\ref{eq:hamiltonians}).

\subsection{The main result}
\label{sec:main-result}

In~\cite{Feigin:1994tc} a maximal commutative subalgebra \( \bethe(\lamb;z)_\mu \) containing \( G(\lamb;z)_\mu \) is constructed. This algebra is called \emph{the Bethe algebra} for \( \Llambmu \).
The affine group \( \Aff_1 \simeq \CC^{\times} \ltimes \CC \) acts on the parameter space by simultaneous scaling and translation on each coordinate, i.e for two scalars \( \alpha,\beta \in \CC \)
\begin{equation*}
  (\alpha,\beta)\cdot z = (\alpha z_1 + \beta,\alpha z_2 + \beta, \ldots,\alpha z_n + \beta).
\end{equation*}
Since the denominators in~(\ref{eq:hamiltonians}) are all of the form \( z_a-z_b \), the Gaudin algebras are invariant under this action, \( G(\lamb;z)_\mu = G(\lamb;\alpha z + \beta)_\mu \). It is also known \( \bethe(\lamb;z)_\mu = \bethe(\lamb;\alpha z + \beta)_\mu \) (see~\cite[Proposition~1]{Chervov:2010fz}).

If \( X_n = \setc{z \in \CC^n}{z_a \neq z_b \text{ for } a \neq b} \), our parameter space becomes \( X_n/\Aff_1 \) which we identify with \( \oM{n+1}(\CC) \), the moduli space of \emph{irreducible} genus \( 0 \) curves with \( n+1 \) marked points, the \( (n+1)^{\text{st}} \) marked point being placed at infinity. We obtain in this way a family of algebras \( \bethe(\lamb)_\mu \) over \( \oM{n+1}(\CC) \). We denote the spectrum of this family by
\begin{equation*}
  \pi : \bspec(\lamb)_\mu \deq \Spec \bethe(\lamb)_\mu \longrightarrow \oM{n+1}(\CC).
\end{equation*}
The morphism \( \pi \) is finite (i.e. it is a finite ramified covering space) and is our main object of study. Our aim will be to say something about the Galois theory of this map. We denote by \( \Gal(\pi) \), the Galois group of \( \pi \) (see Section~\ref{sec:galois-actions}) and by \( \bspec(\lamb;z)_\mu \), the fibre over the point \( z \in \oM{n+1}(\CC) \).

In~\cite{Henriques:2006is} Henriques and Kamnitzer show the \emph{cactus group} \( J_n \) acts on the crystal \( \crys(\lamb) = \crys(\lambda_{1}) \ox \crys(\lambda_{2}) \ox \cdots \ox \crys(\lambda_{n}) \) for \( \crys(\lambda) \) the irreducible crystal of highest weight \( \lambda \). In fact this action preserves weight spaces and singular vectors so restricts to an action on \( \Blambmu \). The cactus group contains a subgroup \( PJ_n \), the \emph{pure cactus group}. The following is the main theorem of the paper. It was conjectured in a paper of Rybnikov~\cite[Conjecture~1.6]{Rybnikov:2014wh} (where it is attributed to Etingof). The statement was also conjectured independently by Brochier-Gordon which is where the author first learnt of the statement.

\begin{Theorem}
  \label{thm:main-thm}
For generic \( z \in \oM{n+1} \) there exists a map \( PJ_n \longrightarrow \Gal(\pi) \) from the pure cactus group to the Galois group of \( \pi \) such that there exists a naturally defined bijection 
\begin{equation*}
  \begin{tikzcd}
    \bspec(\lamb;z)_\mu \arrow{r}{\sim} & \Blambmu,
  \end{tikzcd}
\end{equation*}
equivariant for the induced action of \( PJ_n \) on \( \bspec(\lamb;z)_\mu \).
\end{Theorem}

The author is has been informed Kamnitzer and Rybnikov have obtained similar results. The strategy for proving Theorem~\ref{thm:main-thm} is to use the fact the Bethe algebras are isomorphic to functions on intersections of Schubert varieties. Speyer~\cite{Speyer:2014gg} constructs a compactification of a flat family of Schubert intersections and describes it combinatorially. We use this description to calculate the monodromy of this family and relate it back to the Galois theory of the Bethe algebras.

\subsection{The vector representation and Calogero-Moser space}
\label{sec:vect-rep-CM}

An important case is when the partitions \( \lambda_{s} \) are all equal to \( \sq = (1) \). This means \( V = L(\lambda_{s}) = L(\sq) \) is the vector representation, let \( \crys = \crys(\sq) \) be the associated crystal. The importance stems from the fact \( L(\lambda) \) can be embedded into \( V^{\otimes m} \) for some large enough \( m \). Several times we will reduce to this case in proofs.

Let \( r=n \), \( \lamb = (\sq^n) \) and let \( \mu = (1,1,\ldots,1) \). We can identify the set \( \crys(\lamb)_\mu = [\crys^{\ox n}]_{(1,1\ldots,1)} \) with words of length \( n \) in the letters \( \set{1,2,\ldots,n} \) an thus with the symmetric group \( S_n \). The \( J_n \)-orbits in \(  [\crys^{\ox n}]^{\mathsf{sing}}_{(1,1\ldots,1)} \) are exactly the Kazhdan-Lusztig cells. It is shown in~\cite{Mukhin:2014wc} that the family \( \bspec(\sq^n)_{(1,1,\ldots,1)} \) is a subvariety of (type A) Calogero-Moser space (in fact the entire Calogero-Moser space is realised using slightly more general Bethe algebras). 

Bonnaf\'e and Rouquier~\cite{Bonnafe:2013ug} conjecture and provide evidence for a close link between the geometry of Calogero-Moser space and the Kazhdan-Lusztig theory of the associated Coxeter group in all types. In particular it is conjectured the Kazhdan-Lusztig cells are produced as the orbits of a Galois group action. Theorem~\ref{thm:main-thm} provides evidence that the Kazhdan-Lusztig cells can in fact be recovered from the Galois theory of Calogero-Moser space in type A, which will be developed in a forthcoming work of Brochier, Gordon and the author.

\subsection{Moduli of the Gaudin Hamiltonians}
\label{sec:moduli-gaud-hamilt}

Aguirre, Felder and Veselov~\cite{Aguirre:2011gy} showed the algebras of Hamiltonians \( G(\lamb;z)_\mu \) fit into a family of commutative algebras over \( \M{n+1}(\CC) \). The pure cactus group \( PJ_n \) is the fundamental group of \( \M{n+1}(\RR) \). Thus one would like the limits of the Gaudin algebras described by this moduli to have simple spectrum over the real points and use this covering to calculate the monodromy. Unfortunately these algebras do not always have simple spectrum. The first example where this fails is for \( n = 6 \) and \( \mu = (3,2,1) \).

In the case when \( \lamb = (\sq^n) \), the Gaudin Hamiltonians~(\ref{eq:hamiltonians}) always generate a maximal commutative subalgebra of \( \End([V^{\otimes n}]^{\mathsf{sing}}_\mu) \), with simple spectrum when they are diagonalisable and thus generate the entire Bethe algebra (see~\cite{Mukhin:2010ky}). This means Theorem~\ref{thm:main-thm} also describes the Galois theory of the Gaudin algebras. This is not always true for more general \( \lamb \) and we do not know if there are cases where the Gaudin Hamiltonians fail to generate the Bethe algebra for all \( z \in \oM{n+1}(\CC) \).

\subsection{Outline}
\label{sec:outline}

In Section~\ref{sec:preliminaries} we outline some notation and preliminary notions we will rely on throughout the paper, in particular we will recall some of the combinatorics we need. Section~\ref{sec:speyers-flat-family} recalls the construction of Speyer's flat family and describes an action of the symmetric group on this family. We use this to calculate the equivariant monodromy of Speyer's family. In Section~\ref{sec:bethe-algebras} we recall the definition of the Bethe algebras and  Mukhin, Tarasov and Varchenko's isomorphism between the Bethe algebra and functions on Schubert intersections. We prove that the equivariant monodromy in Speyer's family is given by the action of \( J_n \) on crystals and use this to prove Theorem~\ref{thm:main-thm}.

\subsection{Acknowledgements}
\label{sec:acknowledgements}

The current work was completed while the author was a student at the University of Edinburgh and will form part of his PhD thesis. The author would like to thank his supervisors, Iain Gordon and Michael Wemyss for their guidance, and the School of Mathematics, University of Edinburgh for its support.

\section{Preliminaries}
\label{sec:preliminaries}

\subsection{Notation}
\label{sec:notation}

We collect here some notation used throughout the paper. The set \( \set{1,2,\ldots,n} \) for a positive integer \( n \) will be denoted \( [n] \). The set of partitions will be denoted \( \Part \). We will use \( \sq \) to denote the partition \( (1) \) and \( \Lambda = \Lambda_{r,d} = ((d-r)^r) \), the rectangular partition with \( r \) rows and \( d-r \) columns.

For \( \lambda, \mu \in \Part \), we denote the set of semistandard and standard tableaux for the skew shape \( \lambda \bs \mu \) by \( \SSYT(\lambda\bs \mu) \) or \( \SYT(\lambda\bs\mu) \) respectively. If \( T \in \SSYT(\lambda\bs\mu) \)  denote by \( T|_{r,s} \) the skew tableaux obtained from \( T \) by ignoring boxes labelled with numbers outside the range \( [r,s] \).

We will use the \emph{Sch\"utzenberger involution} many times. It will play differing roles depending on whether we consider it as an involution of semistandard tableaux or standard tableaux. To make this distinction more obvious we will use the notation \( \xi \) only for semistandard tableaux and the notation \( \evac \) when applying the involution to standard tableaux. A definition can be found in~\cite{Fulton:1997vaa}.

\subsection{Equivariant fundamental groups and monodromy}
\label{sec:equiv-fund-groups}

We briefly recall the equivariant fundamental and monodromy group. We follow the definition given by Rhodes~\cite{Rhodes:1966bc}. It is possible to formulate the definitions in many languages, some of which apply in much broader generality (i.e the language of orbifolds or stacks) but since we deal with a reasonably simple situation we will stick with the more explicit language.

\begin{Definition}
  \label{def:equivariant-fun-grp}
Let \( X \) be a topological space and \( G \) a discrete group acting on \( X \). Let \( b \in X \) be a basepoint. The \emph{equivariant fundamental group} \( \pi_1^G(X,b) \) has elements \( (\alpha,g) \) where \( g \in G \) and \( \alpha \) is a (homotopy class of a) path from \( b \) to \( g \cdot b \). The group structure is defined by
\begin{equation*}
  (\alpha, g) \cdot (\beta, h) = (\alpha \cdot g(\beta), gh ).
\end{equation*}
Here we use the usual composition of paths and \( g(\beta) \) to denote the \( g \)-translate of the path \( \beta \). An element of \( \pi_1^G(X,b) \) is called a \( G \)-\emph{equivariant loop}.
\end{Definition}

The fundamental group \( \pi_1(X,b) \) is the kernel of the projection \( \pi_1^G(X,b) \rightarrow G \). If \( f \map{S}{B} \) is a \( G \)-equivariant topological covering we can define an action of the group \( \pi_1^G(X,b) \) on the fibre \( f^{-1}(b) \). If \( p \in f^{-1}(b) \) and \( (\alpha,g) \in \pi_1^G(X,b) \) then denote by \( \tilde{\alpha} \) the unique lift of \( \alpha \) to \( S \) such that \( \tilde{\alpha}(0) = p \). Since \( f \) is \( G \)-equivariant \( \tilde{\alpha}(1) \in f^{-1}(g \cdot b) \). Define \( (\alpha,g) \cdot p  = g^{-1} \cdot \tilde{\alpha}(1) \).

\begin{Definition}
  \label{def:equivariant-monodromy}
The above action is the \emph{\( G \)-equivariant monodromy action} of \( \pi_1^G(B,b) \) on \( f^{-1}(b) \). The image of \( \pi_1^G(B,b) \) in \( S_{f^{-1}(b)} \), the symmetric group on the fibre, is the \emph{equivariant monodromy group} and is denoted \( M^G(f;b) \).
\end{Definition}

\subsection{Tilling of \( \M{k}(\RR) \) by associahedra}
\label{sec:tilling-Mk-associahedra}

In this section we will describe a \( CW \)-structure on the real points \( \M{n+1}(\RR) \). This has been investigated in \cite{Devadoss:1999cz}, \cite{Kapranov:1993kz}, and \cite{Davis:2003hl}. We can define a stratification on \( \M{k}(\RR) \) by subspaces 
\begin{equation*}
   \M{k}(\RR) = M_1 \supset M_2 \supset \ldots \supset M_{k-2} \subset \emptyset
\end{equation*}
where \( M_i \) is the set of stable curves with at least \( i \) irreducible components.

\subsubsection{Circular orderings}
\label{sec:circular-orderings}

Let \( D_k \subset S_k \) be the dihedral group generated by \( (12\ldots n) \) and the involution reversing the order of \( 1,2,\ldots, n \). A circular ordering of the integers \( \set{1,2,\ldots,k} \) is an element of \( S_k/D_k \). That is, we imagine ordering the integers on a circle and identify orderings which coincide upon rotation or reflection. The orderings \( (1,2,3,4) \), \( (4,1,2,3) \) and \( (4,3,2,1) \) all represent the same circular ordering but are distinct from \( (1,3,2,4) \).

The order in which the marked points appear on a curve \( C \in \oM{k}(\RR) \) defines a circular ordering. For each circular order \( s \in S_k/D_k \), let \( \associahedron_s \subseteq \M{k}(\RR) \) be the closure of the subspace of curves with circular ordering \( s \). For example, \( \associahedron_{\id} \) is the closure of the set of irreducible curves projectively equivalent to a curve with marked points \( z_1 < z_2 < \ldots < z_k \). 

By a theorem of Kapranov~\cite[Proposition~4.8]{Kapranov:1993kz}, restricting the stratification to \( \associahedron_s \) gives it the structure of a CW-complex with \( i \)-skeleton \( \associahedron_s \cap M_{k-2-i} \). The symmetric group \( S_k \) acts on \( \M{k}(\RR) \) by permuting marked points. This action transitively permutes the cell complexes \( \associahedron_s \) and preserves \( i \)-cells. They are thus are all isomorphic. We define the \( (k-3) \)-\emph{associahedron} to be this cell complex.

\subsubsection{The fundamental group}
\label{sec:fundamental-group}

The \emph{cactus group}, \( J_n \), is the group with generators \( s_{pq} \) for \( 1 \le p < q \le n \) and relations
\begin{enumerate}[\hspace{1em}(i)]
\item \( s_{pq}^2 = 1 \)
\item \( s_{pq} s_{kl} = s_{kl} s_{pq} \) if the intervals \( [p,q] \) and \( [k,l] \) are disjoint.
\item \( s_{pq} s_{kl} = s_{uv} s_{pq} \) if \( [k,l] \subseteq [p,q] \), where \( v = \hat{s}_{pq}(k) \) and \( u = \hat{s}_{pq}(l) \),
\end{enumerate}
where \( \hat{s}_{pq} \) is the permutation that reverses the order of the interval \( [p,q] \). This also provides a maps to the symmetric group \( S_n \). The \emph{pure cactus group} \( PJ_n \) is defined to be the kernel of this homomorphism.

\begin{Lemma}
  \label{lem:generators-of-J_n}
The cactus group \( J_n \) is generated by the elements \( s_{1q} \) for \( 2 \le q \le n \).
\end{Lemma}

\begin{proof}
By the relations for \( J_n \) given above we have \( s_{pq} = s_{1q}s_{1(q-p+1)}s_{1q} \). Since the elements \( s_{pq} \) generate \( J_n \) so do the \( s_{1q} \).
\end{proof}

In~\cite{Henriques:2006is} it is shown \( \pi_1(\M{n+1}(\RR)) = PJ_n \). The space \( \M{n+1} \) also has an action of \( S_n \) by permuting the first \( n \) marked points. This leaves the real points stable. The equivariant fundamental group \( \pi^{S_n}_1(\M{n+1}(\RR)) \) is \( J_n \). The equivariant loop in \( \M{n+1}(\RR) \) corresponding to \( s_{pq} \in J_n \) is \( (\alpha,\hat{s}_{pq}) \), where \( \alpha \) is the path from a basepoint \( C \) passing through the wall reversing the labels \( p, \ldots, q \) to \( \hat{s}_{pq}\cdot C \).

\subsection{Growth diagrams}
\label{sec:growth-diagrams}

In this section we recall the notion of a growth diagram. Growth diagrams give an interpretation of \emph{jeu de taquin} slides for standard tableaux using combinatorial objects built on subsets of the lattice \( \ZZ^2 \). When we draw this lattice we will depict the second coordinate as increasing northward on the vertical axis and (perhaps counter intuitively) we depict the first coordinate as increasing \emph{westward} on the horizontal axis. This choice is made in order to be consistent with the notation in~\cite{Speyer:2014gg}. 

\subsubsection{Growth diagrams}
\label{sec:growth-diagrams-def}
  
Let \( \II \) be a subset of \( \ZZ^2_+ = \setc{(i,j) \in \ZZ^2}{j-i \ge 0} \). A \emph{growth diagram} on \( \II \) is a map \( \gamma\map{\II}{\Part} \) obeying the following rules:
\begin{enumerate}[(i)]
\item\label{item:j-i=k_partition} If \( j-i = k \ge 0 \) then \( \gamma_{ij} \) is a partition of \( k \).
\item\label{item:add-singe-box} Suppose \( (i,j) \in \II \). Then if \( (i-1,j) \) (respectively \( (i,j+1) \)) is in \( \II \) then \( \gamma_{ij} \subset \gamma_{(i-1)j} \) (respectively \( \gamma_{ij} \subset \gamma_{i(j+1)} \)).
\item\label{item:jdt-condition} If \( (i,j), (i-1,j), (i,j+1) \) and \( (i-1,j+1) \in \II \) and \( \gamma_{(i-1)(j+1)} \bs \gamma_{ij} \) consists of two boxes that \emph{do not} share an edge then \( \gamma_{(i-1)j} \neq \gamma_{i(j+1)} \).
\end{enumerate}

In view of condition~\ref{item:j-i=k_partition}, condition~\ref{item:add-singe-box} means if we move one step north or one step east in \( \II \), we add a single box. Condition~\ref{item:jdt-condition} means if we have an entire square in \( \II \), and if there are two possible ways to go from \( \gamma_{ij} \) to \( \gamma_{(i-1)(j+1)} \) by adding boxes then the two paths around the square should be these two different ways.

A \emph{path} through \( \II \subset \ZZ^2_+ \) is a series of steps from one vertex to another using only northward and eastward moves (i.e. only ever increasing \( j \) and decreasing \( i \) and thus \( j-i \) is a strictly increasing function on the path). Given a growth diagram \( \gamma \) in \( \II \), every path determines a standard tableau. 

Given a rectangular region in a growth diagram, conditions~\ref{item:j-i=k_partition} and~\ref{item:jdt-condition} mean the entire rectangular region is determined by specifying the tableaux along any path from its bottom left to top right corner.

\subsubsection{The Sch\"utzenberger involution in growth diagrams}
\label{sec:schutz-invol-growth}

We now explain how growth diagrams encode the jeu de taquin slides on standard tableaux (see~\cite{Fulton:1997vaa} for a definition). Let \( \II \) be a rectangular region in \( \ZZ^2_+ \) only one step tall. So \( \II = \setc{(i,j)}{i = r, r+1 \ldots, s \text{ and } j = t, t+1} \), for some \( s \) and \( t \) so that \( (s,t) \) is the bottom left hand corner and \( t-s \ge 0 \). Let \( \gamma \)  be a growth diagram on \( \II \) and set \( \mu = \gamma_{st}, \nu = \gamma_{s(t+1)}, \rho = \gamma_{rt} \) and \( \lambda = \gamma_{r(t+1)} \). These are the four partitions at the corners of \( \II \). 

Let \( T \) be the \( \lambda \bs \nu \)-tableau given by the top edge of \( \II \) and \( S \) the \( \rho \bs \mu \)-tableau given by the bottom edge. See Figure~\ref{fig:jdt-growth}. The partition \( \mu \) determines a node, denoted \( \circ \), on the north-western boundary of \( T \) which we can slide into. Similarly the partition \( \lambda \) determines a node on the south-eastern boundary of \( S \) denoted \( \ast \), which we can slide into.

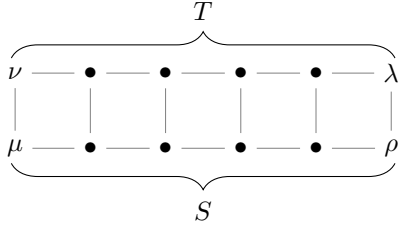
\begin{figure}
  \centering
  \begin{tikzpicture}[scale=1]
    \node (00) at (0,0) {\( \mu \)};
    \node (01) at (0,1) {\( \nu \)};

    \node (50) at (5,0) {\( \rho \)};
    \node (51) at (5,1) {\( \lambda \)};

    \node (10) at (1,0) {\( \bullet \)};
    \node (11) at (1,1) {\( \bullet \)};
    \node (20) at (2,0) {\( \bullet \)};
    \node (21) at (2,1) {\( \bullet \)};
    \node (30) at (3,0) {\( \bullet \)};
    \node (31) at (3,1) {\( \bullet \)};
    \node (40) at (4,0) {\( \bullet \)};
    \node (41) at (4,1) {\( \bullet \)};

    \foreach \from/\to in {00/01,10/11,20/21,30/31,40/41,50/51,00/10,10/20,20/30,30/40,40/50,01/11,11/21,21/31,31/41,41/51} 
        \draw[help lines] (\from) -- (\to); 

    \draw[decorate,decoration={brace,amplitude=10pt,mirror}] (-0.05,-0.2) -- (5.05,-0.2) node[midway,yshift=-18] {\( S \)};
    \draw[decorate,decoration={brace,amplitude=10pt}] (-0.05,1.2) -- (5.05,1.2) node[midway,yshift=18] {\( T \)};
  \end{tikzpicture}
  \caption{The jeu de taquin growth diagram}
  \label{fig:jdt-growth}
\end{figure}

\begin{Proposition}
  \label{prp:grwoth-diagams-jdt}
The tableau \( S \) is the result of the jeu de taquin slide of \( T \) into \( \circ \) and \( T \) is the result of the slide of \( S \) into \( \ast \).
\end{Proposition}

\begin{proof}
See~\cite{Stanley:1999eh}, Proposition~A1.2.7.
\end{proof}

We now explain how this relates to the Sch\"utzenberger involution for standard tableaux. Suppose \( T \in \SYT(\lambda \bs \mu) \) where \( \mu \) is a partition of \( k \) and \( \lambda \bs \mu \) has \( l \) boxes. Let 
\begin{equation*}
\II = \setc{(i,j) \in \ZZ^2_+}{0 \le i \le l, k \le j \le k+l, \text{ and } i+k \le j }.
\end{equation*}
That is, \( \II \) is a triangle with vertices \( (0,k), (l,k+l) \) and \( (0,k+l) \).
Any growth diagram, \( \gamma \), on \( \II \) can be computed recursively if we know the value of \( \gamma \) on either on the horizontal or vertical side of \( \II \). Define the growth diagram \( \gamma_T \) on \( \II \) by setting \( \gamma_T(r,k+r) = \mu \) for any \( 0 \le r \le l \). That is, on the diagonal edge of \( \II \), \( \gamma_T \) is of constant value \( \mu \). We set the sequence of partitions
\begin{equation*}
\gamma_T(l,k+l) \subset \gamma_T(l-1,k+l) \subset \ldots \subset \gamma_T(0,k+l)
\end{equation*}
on the horizontal edge of \( \II \) so they determine the standard tableau \( T \). By the observation above, this determines \( \gamma_T \) on all of \( \II \). As an immediate consequence of the definition and of Proposition~\ref{prp:grwoth-diagams-jdt} we obtain the following corollary.

\begin{Corollary}
\label{cor:growth-schutz}
Let \( \gamma_T \) be the growth diagram above, associated to a standard tableau \( T \in \SYT(\lambda \bs \mu) \). The standard tableau determined by the sequence of partitions
along the vertical edge of \( \II \) is \( \evac(T) \), the Sch\"utzenberger involution of \( T \).
\end{Corollary}

This explains why the Sch\"utzenberger involution is in fact an involution, at least in the case of standard tableaux.

\subsection{Dual equivalence classes}
\label{sec:dual-equiv-classes}

We now describe Haiman's~\cite{Haiman:1992hu} notion of dual equivalence. This is an equivalence relation on skew tableaux, dual to slide equivalence in the sense that it preserves the Q-symbol of a word given by the RSK correspondence.

\begin{Definition}
\label{def:dual-equiv}
Two semistandard skew tableaux, \( T \) and \( T' \) of the same shape, are called \emph{dual equivalent}, denoted \( T \sim_D T' \), if for all meaningful sequence of slides, applying the sequence of slides to \( T \) and \( T' \) results in tableaux of the same shape.
\end{Definition}

\subsubsection{Dual equivalence is local}
\label{sec:dual-equiv-local}

We have the following proposition which tells us that dual equivalence is a \emph{local} operation. That is, we can replace a subtableau with a dual equivalent one and the resulting tableau will be dual equivalent to the original one.

\begin{Proposition}[{\cite[Lemma~2.1]{Haiman:1992hu}}]
  \label{prp:local-dual-equiv}
Suppose \( X \), \( Y \), \( S \) and \( T \) are semistandard tableaux such that \( X \cup T \cup Y \) and \( X \cup S \cup Y \) are semistandard tableaux. If \( S \sim_D T \) then \( X \cup T \cup Y \sim_D X \cup S \cup Y \).
\end{Proposition}

A skew-shape is called \emph{normal} if it has a unique top left corner, and \emph{antinormal} if it has a unique bottom right corner. That is, if it is the north-western or south-eastern part of a rectangle respectively. We have the following important properties of dual equivalence which we will use many times.

\begin{Theorem}
  \label{thm:props-of-dual-equiv}
Dual equivalence has the following properties.
\begin{enumerate}[(i)]
\item\label{item:normal-all-de} All tableaux of a given normal or antinormal shape are dual equivalent.
\item\label{item:slide-de-unique} The intersection of any slide equivalence class and any dual equivalence class is a unique tableaux.
\item\label{item:words-de-if-Q} Two words are dual equivalent if and only if their Q-symbols agree.
\end{enumerate}
\end{Theorem}

\begin{proof}
Properties~\ref{item:normal-all-de}, \ref{item:slide-de-unique} and~\ref{item:words-de-if-Q} are Proposition~2.14, Theorem~2.13 and Theorem~2.12 in~\cite{Haiman:1992hu} respectively.
\end{proof}

\subsubsection{Shuffling dual equivalence classes}
\label{sec:shuffl-dual-equiv}

Given a rectangular growth diagram let \( S_1 \) and \( S_2 \) denote the standard tableaux defined by the western edge and the northern edge respectively and let  \( T_1 \) and \( T_2 \) denote the standard tableaux defined by the southern and eastern edges respectively.

\begin{Proposition}[{\cite[Proposition~7.6]{Speyer:2014gg}}]
  \label{prp:shuffle-dual-equiv}
The dual equivalence classes of \( T_1 \) and \( T_2 \) remain unchanged if we replace either (or both) \( S_1 \) or \( S_2 \) by dual equivalent tableaux.
\end{Proposition}

Let \( \delta_1 \) and \( \delta_2 \) be dual equivalence classes such that the shape of \( \delta_2 \) extends the shape of \( \delta_1 \). Choose representatives \( S_1 \) and \( S_2 \) for \( \delta_1 \) and \( \delta_2 \) respectively. Construct the unique rectangular growth diagram with western and northern edges given by \( S_1 \) and \( S_2 \) respectively. Define \( \varepsilon_1 \) and \( \varepsilon_2 \) to be the dual equivalence classes of the southern and eastern edges respectively. Proposition~\ref{prp:shuffle-dual-equiv} implies that \( \varepsilon_1 \) and \( \varepsilon_2 \) are independent of the representatives \( S_1 \) and \( S_2 \) chosen.

\begin{Definition}
  \label{def:shuffling-classes}
If \( \delta_1 \) and \( \delta_2 \) are, as above, dual equivalence classes such that the shape of \( \delta_2 \) extends the shape of \( \delta_1 \) we say \( (\varepsilon_1,\varepsilon_2)\) are the \emph{shuffle} of  \( (\delta_1, \delta_2) \).
\end{Definition}

\subsubsection{Dual equivalence growth diagrams}
\label{sec:dual-equiv-growth-diagrams}

We now define the notion of dual equivalence for growth diagrams (introduced in~\cite{Speyer:2014gg}). First we fix an function \( m\map{\ZZ}{\ZZ_{>0}} \) which we call the \emph{interval}. We define several auxiliary functions using \( m \). We define a function \( \hat{m} : \ZZ \longrightarrow \ZZ  \)
\begin{align*}
  \hat{m}(i) &\deq
                 \begin{cases}
                   1+\sum_{k=1}^{i-1} m(k) & \text{if } i > 0 \\
                   1-\sum_{k=i}^0 m(k) & \text{if } i \le 0,
                 \end{cases} \\
\intertext{a function \( \bar{m} : \ZZ^2_+ \longrightarrow \ZZ^2_+ \), \( \bar{m}(i,j) \deq \left( \hat{m}(i),\hat{m}(j) \right) \) and a function \(   m_s: \ZZ^2_+ \longrightarrow \NN  \)}
  m_s(i,j) &\deq \hat{m}(j) - \hat{m}(i) = \sum_{k=i}^{j-1} m(k).
\end{align*}
In particular \( m_s(i,i) = 0 \), \( m_s(i,i+1) = m(i) \) and if \( m \) is the constant function \( 1 \) then \( m_s(i,j) = j-i \). 

Consider the graph with vertices \( \ZZ^2_+ \) and edges \( a_{ij} \) between \( (i,j) \) and \( (i-1,j) \) and edges \( b_{ij} \) between vertices \( (i,j) \) and \( (i,j+1) \). If we embed \( \ZZ^2_+ \subset \RR^2 \), these as simply the horizontal and vertical unit intervals between the points of \( \ZZ^2_+ \).  If \( \II \subset \ZZ^2_+ \), we call an edge of \( \ZZ^2_+ \) \emph{internal} to \( \II \) if both of its endpoints are in \( \II \). 

\begin{Definition}
\label{def:dual-equiv-growth}
  A \emph{dual equivalence growth diagram} on \( \II \) with interval \( m \) is a map \( \gamma\map{\II}{\Part} \) as well as an assignment of a dual equivalence class \( \alpha_{ij} \) (respectively \( \beta_{ij} \)) to every edge \( a_{ij} \) (respectively \( b_{ij} \)) internal to \( \II \), obeying the following rules.
  \begin{enumerate}[(i)]
  \item\label{item:j-i=mij_partition} If \( (i,j) \in \II \) then \( \gamma_{ij} \) is a partition of \( m_s(i,j) \).
  \item\label{item:add-controlled-no-boxes} If \( a_{ij} \) (respectively \( b_{ij} \)) is internal to \( \II \) then \( \gamma_{ij} \subset \gamma_{(i-1)j} \) (resp. \( \gamma_{ij} \subset \gamma_{i(j+1)} \)).
  \item\label{item:dual-equiv-jdt-condition} If \( a_{ij}, b_{(i-1)j}, b_{ij} \) and \( a_{i(j+1)} \in \II \) then \( (\alpha_{ij},\beta_{(i-1)j}) \) is the shuffle of \( (\beta_{ij},\alpha_{i(j+1)}) \).
  \end{enumerate}
\end{Definition}

We should think of the interval \( m \) as defining how many boxes we are allowed to add with each step though the lattice. Indeed, if we wish to move one step east from \( (i,j) \), \( m_s \) increases by \( m(i-1) \) and if we wish to move one step north, \( m_s \) increases by \( m(j) \). Thus the partition \( \gamma_{ii} \) must be the empty partition and \( \gamma_{i(i+1)} \) is a partition of \( m(i) \). This means when \( m \) is the constant function \( m(i) = 1 \) our definition coincides with that for a ordinary growth diagram for \( \II \). 

We can also think of dual equivalence growth diagrams as equivalence classes of certain growth diagrams. Let \( \tilde{\gamma} \) be a growth diagram on \( \tilde{\II} \subset \ZZ^2_+ \). We say \( \tilde{\II} \) is \emph{adapted} to an interval \( m\map{\ZZ}{\ZZ_{>0}} \) if it has the following property: If \( \tilde{\II} \) contains each of the four vertices \( \bar{m}(i,j), \bar{m}(i-1,j), \bar{m}(i,j+1), \) and \( \bar{m}(i-1,j+1) \), then \( \tilde{\II} \) contains all of the vertices in the rectangular region they bound.
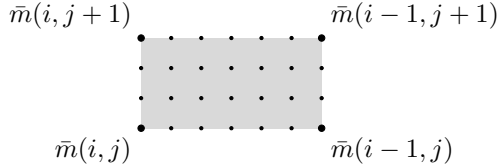
\begin{center}
  \begin{tikzpicture}[scale=0.4]
    \fill[gray,opacity=0.3] (-3,0) rectangle (3,3);
    \foreach \i in {-3,-2,-1,0,1,2,3} {
      \foreach \j in {0,1,2,3} { 
        \draw[fill] (\i,\j) circle (1.5pt);
      } 
    }
    \node[anchor=north east] at (-3,0) {\( \bar{m}(i,j) \)};
    \node[anchor=north west] at ( 3,0) {\( \bar{m}(i-1,j) \)};
    \node[anchor=south east] at (-3,3) {\( \bar{m}(i,j+1) \)};
    \node[anchor=south west] at ( 3,3) {\( \bar{m}(i-1,j+1) \)};

    \draw[fill] (-3,0) circle (3pt);
    \draw[fill] ( 3,0) circle (3pt);
    \draw[fill] (-3,3) circle (3pt);
    \draw[fill] ( 3,3) circle (3pt);
  \end{tikzpicture}
\end{center}

\begin{Definition}
  \label{def:reduction-mod-m}
  If \( \tilde{\II} \) is adapted to \( m \), the \emph{reduction modulo \( m \)} of a growth diagram \( \tilde{\gamma} \) is defined to be the map \( \gamma\map{\II}{\Part} \) for
  \begin{equation*}
    \II = \setc{(i,j) \in \ZZ^2_+}{\left( \hat{m}(i),\hat{m}(j) \right) \in \tilde{\II}},
  \end{equation*}
given by \( \gamma = \tilde{\gamma} \circ \bar{m} \), along with the set of dual equivalence classes
\begin{itemize}
\item  \( \alpha_{ij} \), the dual equivalence class defined by the horizontal path, in \( \tilde{\gamma} \), from \( \bar{m}(i,j) \) to \( \bar{m}(i-1,j) \) and
\item \( \beta_{ij} \), the dual equivalence class of the tableaux defined by the vertical path, in \( \tilde{\gamma} \), from \( \bar{m}(i,j) \) to \( \bar{m}(i-1,j) \).
\end{itemize}
\end{Definition}

\begin{Proposition}
  \label{prp:reduction-mod-m}
The map \( \gamma\map{\II}{\Part} \) along with the choice of \( \alpha_{ij} \) and \( \beta_{ij} \) define a dual equivalence growth diagram on \( \II \).
\end{Proposition}

\begin{proof}
We must check the conditions in Definition~\ref{def:dual-equiv-growth}. For condition \ref{item:j-i=mij_partition} note that \( \gamma_{ij} = \tilde{\gamma}_{\hat{m}(i),\hat{m}(j)} \) so \( \abs{\gamma_{ij}} = \hat{m}(j) - \hat{m}(i) \) which is \( m_s(i,j) \) by definition. We have a path in \( \tilde{\II} \) from \( \bar{m}(i,j) \) to \( \bar{m}(i-1,j) \) so
\begin{equation*}
  \gamma_{ij} = \tilde{\gamma}_{\bar{m}(i,j)} \subset \tilde{\gamma}_{\bar{m}(i-1,j)} = \gamma_{(i-1)j}.
\end{equation*}
Similarly \( \gamma_{ij} \subset \gamma_{i(j+1)} \). To see condition \ref{item:dual-equiv-jdt-condition}, note that \( \alpha_{ij}, \beta_{(i-1)j}, \beta_{ij} \) and \( \alpha_{i(j+1)} \) are defined as the dual equivalence classes coming from the four sides of a rectangular growth diagram:
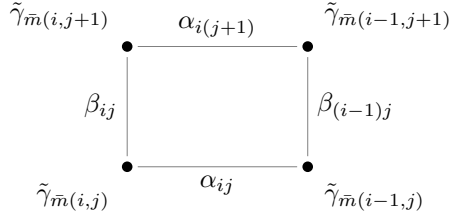
\begin{center}
  \begin{tikzpicture}[scale=0.8]
    \node[black,circle,fill,inner sep=0,outer sep=2,minimum size=4pt,label={[black]225:\( \tilde{\gamma}_{\bar{m}(i,j)} \)}] (sw) at (0,0) {}; 
    \node[black,circle,fill,inner sep=0,outer sep=2,minimum size=4pt,label={[black]315:\( \tilde{\gamma}_{\bar{m}(i-1,j)} \)}] (se) at (3,0) {}; 
    \node[black,circle,fill,inner sep=0,outer sep=2,minimum size=4pt,label={[black]135:\( \tilde{\gamma}_{\bar{m}(i,j+1)} \)}] (nw) at (0,2) {}; 
    \node[black,circle,fill,inner sep=0,outer sep=2,minimum size=4pt,label={[black] 45:\( \tilde{\gamma}_{\bar{m}(i-1,j+1)} \)}] (ne) at (3,2) {}; 

    \draw[help lines] (sw) -- (se); 
    \draw[help lines] (sw) -- (nw); 
    \draw[help lines] (nw) -- (ne); 
    \draw[help lines] (se) -- (ne);

    \node[anchor=east] at (0,1) {\( \beta_{ij} \)};
    \node[anchor=north] at (1.5,0) {\( \alpha_{ij} \)};
    \node[anchor=west] at (3,1) {\( \beta_{(i-1)j} \)};
    \node[anchor=south] at (1.5,2) {\( \alpha_{i(j+1)} \)};
  \end{tikzpicture}
\end{center}
The fact that this portion of \( \tilde{\II} \) forms a rectangular growth diagram is given by the requirement that \( \tilde{\II} \) is adapted to \( m \). By definition, this means \( (\beta_{ij}, \alpha_{i(j+1)}) \) is the shuffle of \( (\alpha_{ij}, \beta_{(i-1)j}) \) as required.
\end{proof}

\section{Speyer's flat family}
\label{sec:speyers-flat-family}

In this section we describe the main geometrical tool that will be used to prove Theorem~\ref{thm:main-thm}. In the paper~\cite{Mukhin:2009et}, Mukhin, Tarasov and Varchenko describe a relationship between Bethe algebras and Schubert calculus. Speyer~\cite{Speyer:2014gg} constructs a flat family of Schubert intersections over \( \M{k}(\CC) \) which, in Section~\ref{sec:bethe-algebras} will be related to the spectrum of the Bethe algebras. We define an action of the symmetric group on the family and Speyer's explicit combinatorial description of the real points to calculate the equivariant monodromy action.

\subsection{Osculating flags}
\label{sec:osculating-flags}

In this section we recall some definitions and facts from Schubert calculus. All the Grassmannians we consider will be defined relative to some genus \( 0 \) smooth curve \( C \). To set this up, choose a very ample line bundle \( \Ll \) on \( C \) of degree \( d-1 \). We have the Veronese embedding
\begin{equation*}
  \varepsilon \map{C}{\PP H^0(C,\Ll)^*}.
\end{equation*}
A point \( p \) is sent by \( \varepsilon \) to the hyperplane of sections vanishing at \( p \). Let
\begin{equation*}
  \Gr(r,d)_C \deq \Gr(r,H^0(C,\Ll)).
\end{equation*}
We can also define the \( r^{\text{th}} \) \emph{associated curve} \( \varepsilon_r \map{C}{\Gr(r,d)_C} \) which sends a point \( p \) to the space of sections vanishing to order at least \( d-r \) at \( p \). That is
\begin{equation*}
  \varepsilon_{r}(p) \deq H^0(C,\Ii_p^{d-r} \ox \Ll) \subset H^0(C,\Ll),
\end{equation*}
Here \( \Ii_p \) is the ideal sheaf of the point \( p \). With this notation \( \varepsilon = \varepsilon_{d-1} \). 

\begin{Definition}
  \label{def:osculating-flag}
The flag \( \Ff_\bullet(p) \) defined by \( \Ff_i(p) = \varepsilon_i(p) \) is called the \emph{osculating flag} at \( p \).
\end{Definition}

\begin{Example}
  \label{exm:mtv-osculating-flags}
We can make this concrete by considering the case \( C = \PP^1 = \PP\CC^2 \). Fix the standard homogeneous coordinates \( [x:y] \) on \( \PP^1 \). Choose the line bundle \( \Oo_{\PP^1}(d-1) \). Then \( H^0(\PP^1,\Oo(d-1)) = \CC[x,y]_{d-1} \), the homogeneous polynomials of degree \( d-1 \). If we work in the affine patch where \( y \neq 0 \) then we identify identify this with \( \CC_d[u] \), the space of polynomials of degree \emph{strictly less than} \( d \) (\( u \) is the coordinate on this patch). The Grassmannian \( \Gr(r,d)_{\PP^1} \) is then the set of \( r \)-dimensional subspaces of \( \CC_d[u] \). The map \( \varepsilon_r \) sends the point \( [b:1] \in \PP^1 \) to the subspace \( (u-b)^{d-r} \CC_{r}[u] \) and the osculating flag \( \Ff_\bullet(b) \) is
\begin{equation*}
  (x-b)^{d-1} \CC_1[x] \subset (x-b)^{d-2} \CC_2[x] \subset \ldots \subset (x-b) \CC_{(d-1)}[x] \subset \CC_d[x].
\end{equation*}
The flag \( \Ff_\bullet(\infty) \) is
\begin{equation*}
  \CC_0[x] \subset \CC_1[x] \subset \ldots \subset \CC_{(d-1)}[x] \subset \CC_d[x].
\end{equation*}
\end{Example}

\begin{Remark}
  \label{rem:drop-notaton-PP1}
When we are in the situation of Example~\ref{exm:mtv-osculating-flags} we will drop the subscript \( \PP^1 \). and write \( \Gr(r,d) \) instead of \( \Gr(r,d)_{\PP^1} \).
\end{Remark}

Suppose we have pairs \( (C,\Ll) \) and \( (D,\Kk) \) of curves and very ample line bundles as well as an isomorphism \( \phi\map{C}{D} \) such that \( \phi_* \Ll \cong \Kk \). We would like some relation between the Grassmannian and osculating flags on each curve. It is important to note it is not possible to choose a canonical isomorphism between \( \phi_*\Ll \) and \( \Kk \), however we have the following fact.

\begin{Lemma}
  \label{lem:iso-inv-sheaves}
Let \( \Ee \) be an invertible \( \Oo_X \)-module for a projective \( \CC \)-scheme \( X \). Then \( \End(\Ee) \cong \CC \).
\end{Lemma}

\begin{proof}
Note that \( \End(\Oo_X) \cong \CC \). The lemma follows from the fact \( \Oo_X \cong \Ee \ox \Ee^* \) and the hom-tensor adjunction formula:
\begin{align*}
  \CC \cong \Hom(\Oo_X,\Oo_X) \cong& \Hom(\Ee \ox \Ee^*, \Oo_X) \\
                              \cong& \Hom(\Ee, \SheafHom(\Ee^*,\Oo_X)) \\
                              \cong& \Hom(\Ee,\Ee). \qedhere
\end{align*}
\end{proof}

This means the isomorphism \( \phi_*\Ll \cong \Kk \) is unique up to scalar multiple. By noting \( H^0(C,-) = H^0(D,\phi_* -) \), we have a canonical induced isomorphism
\begin{equation*}
  \phi_1 : \PP H^0(C,\Ll) \longrightarrow \PP H^0(D,\Kk),
\end{equation*}
as well as canonical isomorphisms
\begin{equation*}
  \phi_r : \Gr(r,d)_C \longrightarrow \Gr(r,d)_D
\end{equation*}
for any \( r \).

\begin{Lemma}
  \label{lem:commutes-with-embedding}
The isomorphism \( \phi_r \) preserves the associated curves, more precisely, \( \phi_r \circ \varepsilon_r = \varepsilon_r \circ \phi \).
In particular \( \Ff_i(\phi(p)) = \phi_r(\Ff_i(p)) \).
\end{Lemma}

\begin{proof}
Choose an isomorphism \( \psi\map{\phi_*\Ll}{\Kk} \). We thus obtain an isomorphism \( H^0(C,\Ll) \longrightarrow H^0(D,\Kk) \)
which we also denote by \( \psi \). Let \( p \in C \) and let \( q = \phi(p) \). We need to show that the image of \( H^0(C,\Ii_p^{d-r}\ox\Ll) \) under \( \psi \) is \( H^0(D,\Ii_q^{d-r} \ox \Kk) \). This follows since \( \psi \) is a module homomorphism and thus sends \( \phi_*(\Ii_p^{d-r}\ox\Ll) \), considered as a submodule of \( \phi_*(\Ll) \), to \( \Ii_q^{d-r}\ox\Kk \).
\end{proof}

\subsection{Schubert intersections}
\label{sec:schub-intersects}

For a partition \( \lambda \), with at most \( r \) rows and \( d-r \) columns and a point \( p \in C \) we will denote the Schubert variety corresponding to the osculating flag \( \Ff_\bullet(p) \) by \( \Omega(\lambda;p)_C \). Let \( \lambda^\cc \) be the partition \emph{complementary} to \( \lambda \) for \( \Gr(r,d)_C \). That is, \( \lambda^\cc \) is obtained from \( \Lambda_{r,d} \bs \lambda \) by rotating \( 180 \) degrees.

\begin{Lemma}
  \label{lem:isomorphism-of-schuberts}
Let \( (C,\Ll), (D,\Kk) \) and \( \phi_r \) be as in Section~\ref{sec:osculating-flags}. The image of \( \Omega(\lambda;p)_C \) under \( \phi_r \) is \( \Omega(\lambda;\phi(p))_D \).
\end{Lemma}

\begin{proof}
Choose an isomorphism \( \psi\map{\phi_*\Ll}{\Kk} \). If \( \Ff \) is a flag in \( H^0(C,\Ll) \) then for any subspace \( V \subset H^0(C,\Ll) \)
\begin{equation*}
  \dim V \cap \Ff_i = \dim \psi(V) \cap \psi(\Ff_i).
\end{equation*}
Thus \( \psi_r(\Omega(\lambda,\Ff)_C) = \Omega(\lambda,\psi(\Ff)) \). By Lemma~\ref{lem:commutes-with-embedding} \( \psi(\Ff(p)) = \Ff(\phi(p)) \).
\end{proof}

\subsection{Speyer's compactification}
\label{sec:speyer-compactification}

In this section, we fix the following data
\begin{itemize}
\item positive integers \( d,r \) and \( k \) such that \( d \le r \), and
\item a sequence of partitions \( \lamb = (\lambda_{1},\lambda_{2},\ldots,\lambda_{k}) \) such that each \( \lambda_i \) has at most \( r \) rows and \( d-r \) columns and  and \( k \le r(d-r) \).
\end{itemize}
Often we will reduce to the case when \( \lamb = (\sq^k) \), i.e. \( \lambda_i = \sq \) for all \( i \in [k] \). We will refer to this as the \emph{fundamental case}.
We are interested in Schubert intersections relative to the conditions \( \lamb \). For \( z = (z_1,z_2,\ldots,z_n) \in (\PP^1)^k \) distinct define
\begin{equation*}
  \Omega(\lamb;z) \deq \Omega(\lambda_{1};z_1) \cap \Omega(\lambda_{2};z_2) \cap \ldots \cap \Omega(\lambda_{k};z_k).
\end{equation*}
After accounting for an action of \( \mathrm{PGL}_2 \) we consider the family \( \rho\map{\Omega(\lamb)}{\oM{k}} \) whose fibre over the curve with marked points \( z \) is \( \Omega(\lamb;z) \). We recall the construction of Speyer's flat families \( \Gg(r,d), \) and \( \Ss(\lambda_{\bullet}) \) from~\cite{Speyer:2014gg} which extend the families \( \Gr(r,d)\times \oM{k} \) and \( \Omega(\lamb) \).

\subsubsection{The construction}
\label{sec:construction}

If \( A \) is a three element subset of \( [k] \), fix a curve \( C_T \) isomorphic to \( \PP^1 \) with three points marked by the elements of \( A \). Since \( A \) consists of exactly three elements, the choice of \( C_A \) is unique up to projective equivalence. We write \( \Gr(r,d)_A \) for \( \Gr(r,d)_{C_A} \). For a curve \( C \in \oM{k} \) with marked points \( (z_1,\ldots,z_k) \) let \( \phi_A(C) \map{\PP^1}{C_A} \) be the unique isomorphism that for each \( a \in A \) sends \( z_a \in \PP^1 \) to the point on \( C_A \) marked by \( a \). In this way we obtain a morphism
\begin{equation*}
  \phi_A\map{\oM{k}\times \PP^1}{\oM{k} \times C_A}.
\end{equation*}
Applying the Grassmannian construction to the family of curves we obtain a morphism
\begin{equation*}
  \phi_A\map{\oM{k}\times \Gr(r,d)_{\PP^1}}{\oM{k} \times \Gr(r,d)_A}.
\end{equation*}
Using these morphisms we construct an embedding 
\begin{equation*}
  \begin{tikzcd}[row sep = 0pt]
      \oM{k} \times \Gr(r,d) \arrow[hook]{r} & \M{k} \times \prod_A \Gr(r,d)_A \\
      (C,X) \arrow[mapsto]{r} & (C,\phi_A(C,X)).
  \end{tikzcd}
\end{equation*}
Identify \( \oM{k} \times \Gr(r,d) \) with its image in \( \M{k} \times \prod_A \Gr(r,d)_A \).

\begin{Definition}[{\cite{Speyer:2014gg}}]
  \label{def:Speyer-families}
The family \( \Gg(r,d) \) is the closure of \(\oM{k} \times \Gr(r,d) \) in \( \M{k} \times \prod_A \Gr(r,d)_A \). Also define the subvariety \( \Ss(\lamb) \) as \( \Gg(r,d) \cap \bigcap_{a \in A}\Omega(\lambda_{a},a)_A \).
\end{Definition}

\begin{Theorem}[{\cite[Theorem~1.1]{Speyer:2014gg}}]
  \label{thm:props-of-G-and-S}
The family \( \Gg(r,d) \) and its subfamily \( \Ss(\lamb) \) have the following properties:
\begin{enumerate}[(i)]
\item \( \Gg(r,d) \) and \( \Ss(\lamb) \) are Cohen-Macaulay and flat over \( \M{k} \).
\item \( \Gg(r,d) \) is isomorphic to \( \Gr(r,d) \times \oM{k} \) over \( \oM{k} \).
\item \( \Ss(\lamb) \) is isomorphic to \( \Omega(\lambda_{\bullet}) \) over \( \oM{k} \).
\item If a representative of \( C \in \oM{k} \) has marked points \( z_1, z_2, \ldots z_k \in \PP^1 \) then the fibre of \( \Ss(\lamb) \) over \( C \) is isomorphic to \( \bigcap \Omega(\lambda_i,z_i) \)
\end{enumerate}
\end{Theorem}

\begin{Theorem}[{\cite[Theorem~1.4]{Speyer:2014gg}}]
\label{thm:real-points-labelling}
If \( \abs{\lamb} = \sum \abs{\lambda_i} = r(d-r) \), the fibre of \( \Ss(\lambda_{\bullet}) \) over \( C \in \M{k}(\RR) \) is a reduced union of real points. 
\end{Theorem}

\subsubsection{The fibre}
\label{sec:fibre}

We will also want an explicit description of the fibres of \( \Ss(\lamb) \) so let us recall this from~\cite{Speyer:2014gg}. Fix \( C \in \M{k} \), a not necessarily irreducible curve and denote its irreducible components \( C_1, C_2, \ldots C_l \). Fix an irreducible component \( C_i \) and let \( A \subseteq [k] \) be a three element subset. If \( d_1,\ldots,d_e \) are the nodes lying on \( C_i \) we say that \( v(A) = C_i \) if the points marked by \( A \) lie on three separate connected components of \( C \bs \set{d_1,\ldots,d_e} \).

Define the \emph{projection of} \( a \in [n] \) onto \( C_i \): if \( a \) marks a point on \( C_i \) then then projection is this point, otherwise there is a unique node \( d \in C_i \) via which \( a \) is path connected to \( C_i \), let \( d \) be the projection. If \( v(A) = C_i \) then the projection of \( A \) onto \( C_i \) produces three distinct points on \( C_i \).

If \( v(A) = C_i \) define the isomorphism \( \phi_{i,A}\map{C_i}{C_A} \) given sending the projection onto \( C_i \) of \( a \in A \) to the point marked by \( a \) in \( C_A \). This morphism is uniquely determined since \( v(A)=C_i \). By considering the corresponding isomorphisms \( \Gr(r,d)_{C_i} \longrightarrow \Gr(r,d)_A \) we obtain an embedding
\begin{equation*}
  \begin{tikzcd}
    \Gr(r,d)_{C_i} \arrow[hook]{r} & \prod_{v(A)=C_i} \Gr(r,d)_A.
  \end{tikzcd}
\end{equation*}
We will identify \( \Gr(r,d)_{C_i} \) with its image. Speyer shows  the projection from \( \Gg(r,d) \) into \( \prod_{v(A)=C_i} \Gr(r,d)_A \) lands inside \( \Gr(r,d)_{C_i} \). In this way we can think of the fibre \( \Gg(r,d)(C) \) as a subvariety of \( \prod_{i} \Gr(r,d)_{C_i} \).

\begin{Definition}
  \label{def:node-labelling}
A \emph{node labelling} for \( C \) is a function \( \nu \)  which assigns to every pair \( (C_i,d) \), of an irreducible component and node \( d \in C_j \), a partition \( \nu(C_j,x) \) such that if \( d \in C_i \cap C_j \) then \( {\nu(C_i,d)}^\cc = \nu(C_j,d) \). Denote the set of node labellings by \( \mathtt{N}_C \)
\end{Definition}

\begin{Theorem}[{\cite[Section~3, proof of Theorem~1.2]{Speyer:2014gg}}]
  \label{thm:fibre-theorem}
Let \( C \in \M{k} \) be a stable curve with irreducible components \( C_1, C_2, \ldots, C_l \). Let \( D_i \) be the set of nodes on \( C_i \) and \( P_i \) the set of marked points.
The fibres of \( \Gg(r,d) \) and \( \Ss(\lamb) \) over \( C  \) are
\begin{align}
  \Gg(r,d)(C) &= \bigcup_{\nu \in \ttN_C} \prod_{i} \bigcap_{d \in D_i} \Omega(\nu(C_i,d),d)_{C_i}, \label{eq:fibre-of-G} \\
  \Ss(\lambda_{\bullet})(C) &= \bigcup_{\nu \in \ttN_C} \prod_{i} \left( \bigcap_{d \in D_i} \Omega(\nu(C_i,d),d)_{C_i}  \cap \bigcap_{p \in P_i} \Omega(\lambda_{p},p)_{C_i} \right). \label{eq:fibre-of-S}
\end{align}
\end{Theorem}

\subsection{Speyer's labelling of the fibre}
\label{sec:spey-labell-fibre}

Restrict to the case when \( \abs{\lamb} = r(d-r) \). By Theorem~\ref{thm:real-points-labelling} the family \( \Ss(\lamb)(\RR) \longrightarrow \M{k}(\RR) \) is a topological covering of degree \( c^{\Lambda}_{\lamb} \). Recall from Section~\ref{sec:tilling-Mk-associahedra} that \( \M{k}(\RR) \) is tiled by associahedra. This tiling is indexed by circular orderings of the set \( [k] \). We can lift the cellular structure to a tiling by associahedra of \( \Ss(\lamb)(\RR) \) and the aim of this section will be to explain Speyer's combinatorial description of this CW-complex structure.

For now, we will just consider the fundamental case \( \Ss(\sq^k) \). Choose a circular ordering, \( s = (s(1),s(2),\ldots,s(k)) \), and let \( \associahedron \) be an associahedron of \( \Ss(\sq^k)(\RR) \) lying above the associahedron corresponding to \( s \) in \( \M{k}(\RR) \). The associahedron has facets labelled by non-adjacent pairs \( (i,j) \) where \( i < j \). The facet \( \associahedron_{ij} \) of \( \associahedron \) lies over stable curves that generically have two components, one containing (in order) the labels \( s(i), s(i+1), \ldots, s(j-1) \) and the other containing the labels \( s(j), s(j+1), \ldots, s(i-1) \). Such a stable curve is depicted in Figure~\ref{fig:stab-curve-theta_ij}.

\begin{figure}[t]
  \centering
  \begin{tikzpicture}
    \draw[help lines,postaction={decorate,decoration={
      markings,
      mark=at position   90/360 with {\node[black,circle,fill,inner sep=0,minimum size=0pt,label={[black] 90:$\ldots$}] (dots) {};},
      mark=at position  140/360 with {\node[black,circle,fill,inner sep=0,minimum size=4pt,label={[black]150:$s(i+1)$}] (i+1) {};},
      mark=at position  200/360 with {\node[black,circle,fill,inner sep=0,minimum size=4pt,label={[black]180:$s(i)$}] (i) {};},
      mark=at position  340/360 with {\node[black,circle,fill,inner sep=0,minimum size=4pt,label={[black]350:$s(j-1)$}] (j-1) {};}
  }}]
 (0,0) circle (20pt);

\draw[help lines,postaction={decorate,decoration={
      markings,
      mark=at position   10/360 with {\node[black,circle,fill,inner sep=0,minimum size=4pt,label={[black]0:$s(j)$}] (j) {};},
      mark=at position  170/360 with {\node[black,circle,fill,inner sep=0,minimum size=4pt,label={[black]180:$s(i-1)$}] (i-1) {};},
      mark=at position  310/360 with {\node[black,circle,fill,inner sep=0,minimum size=4pt,label={[black]330:$s(j+1)$}] (j+1) {};},
      mark=at position  270/360 with {\node[black,circle,fill,inner sep=0,minimum size=0pt,label={[black]270:$\ldots$}] (dots2) {};}
  }}]
(0,-40pt) circle (20pt);

\node[anchor=south] at (0,-20pt) {\( \gamma_{ij} \)};
\node[anchor=north] at (0,-20pt) {\( \gamma_{ij}^\cc \)};
  \end{tikzpicture}
  \caption{A stable curve in $ \associahedron_{ij} $.}
  \label{fig:stab-curve-theta_ij}
\end{figure}

Fix a generic point \( C \) in \( \associahedron_{ij} \). Theorem~\ref{thm:fibre-theorem} tells us that the map \( \nu \) assigns a partition to either side of the node of the stable curve at \( C \). Let \( \gamma_{ij} \) be the partition assigned to the side of the node \emph{away from} the component labelled \( s(i), s(i+1), \ldots, s(j-1) \). Again see Figure~\ref{fig:stab-curve-theta_ij} for a depiction of this situation. In fact, \( \gamma_{ij} \) does not depend on \( x \) (see \cite[Lemma 7.1]{Speyer:2014gg}).

\subsubsection{Cylindrical growth diagrams}
\label{sec:cgds}

Let \( \mathbb{I} = \setc{(i,j) \in \ZZ^2}{0 \le j-i \le k} \). 

\begin{Definition}
\label{def:cylindrical-growth-diagram}
A growth diagram on \( \II \) is a \emph{cylindrical growth diagram} for \( (r,d) \) if it satisfies the condition that \( \gamma_{i(i+k)} = \Lambda \).
\end{Definition}

The reason for the adjective cylindrical is that these growth diagrams are periodic along the north-west diagonal. An example of (part of) a cylindrical growth diagram for \( r = 2 \) and \( d = 5 \) is given in Figure~\ref{fig:example-cgd}. As one can see from the diagram, the bottom row is repeated at the top. 

\begin{figure}[t]
  \centering
  \ytableausetup{nosmalltableaux}
  \ytableausetup{boxsize=0.4em}
  \begin{tikzpicture}[scale=0.95]
    \node (11)  at (-1,1) {\( \emptyset \)};
    \node (01)  at (0,1)  {\ydiagram{1}};
    \node (-11) at (1,1)  {\ydiagram{2}};
    \node (-21) at (2,1)  {\ydiagram{2,1}};
    \node (-31) at (3,1)  {\ydiagram{2,2}};
    \node (-41) at (4,1)  {\ydiagram{3,2}};
    \node (-51) at (5,1)  {\ydiagram{3,3}};

    \node (22)  at (-2,2) {\( \emptyset \)};
    \node (12)  at (-1,2) {\ydiagram{1}};
    \node (02)  at (0,2)  {\ydiagram{2}};
    \node (-12) at (1,2)  {\ydiagram{3}};
    \node (-22) at (2,2)  {\ydiagram{3,1}};
    \node (-32) at (3,2)  {\ydiagram{3,2}};
    \node (-42) at (4,2)  {\ydiagram{3,3}};

    \node (33)  at (-3,3) {\( \emptyset \)};
    \node (23)  at (-2,3) {\ydiagram{1}};
    \node (13)  at (-1,3) {\ydiagram{1,1}};
    \node (03)  at (0,3)  {\ydiagram{2,1}};
    \node (-13) at (1,3)  {\ydiagram{3,1}};
    \node (-23) at (2,3)  {\ydiagram{3,2}};
    \node (-33) at (3,3)  {\ydiagram{3,3}};

    \node (44)  at (-4,4) {\( \emptyset \)};
    \node (34)  at (-3,4) {\ydiagram{1}};
    \node (24)  at (-2,4) {\ydiagram{2}};
    \node (14)  at (-1,4) {\ydiagram{2,1}};
    \node (04)  at (0,4)  {\ydiagram{2,2}};
    \node (-14) at (1,4)  {\ydiagram{3,2}};
    \node (-24) at (2,4)  {\ydiagram{3,3}};

    \node (55)  at (-5,5) {\( \emptyset \)};
    \node (45)  at (-4,5) {\ydiagram{1}};
    \node (35)  at (-3,5) {\ydiagram{2}};
    \node (25)  at (-2,5) {\ydiagram{3}};
    \node (15)  at (-1,5) {\ydiagram{3,1}};
    \node (05)  at (0,5)  {\ydiagram{3,2}};
    \node (-15) at (1,5)  {\ydiagram{3,3}};

    \node (66)  at (-6,6) {\( \emptyset \)};
    \node (56)  at (-5,6) {\ydiagram{1}};
    \node (46)  at (-4,6) {\ydiagram{1,1}};
    \node (36)  at (-3,6) {\ydiagram{2,1}};
    \node (26)  at (-2,6) {\ydiagram{3,1}};
    \node (16)  at (-1,6) {\ydiagram{3,2}};
    \node (06)  at (0,6)  {\ydiagram{3,3}};

    \node (77)  at (-7,7) {\( \emptyset \)};
    \node (67)  at (-6,7) {\ydiagram{1}};
    \node (57)  at (-5,7) {\ydiagram{2}};
    \node (47)  at (-4,7) {\ydiagram{2,1}};
    \node (37)  at (-3,7) {\ydiagram{2,2}};
    \node (27)  at (-2,7) {\ydiagram{3,2}};
    \node (17)  at (-1,7) {\ydiagram{3,3}};

    \foreach \from/\to in {
11/01, 01/-11, -11/-21, -21/-31, -31/-41, -41/-51,  
22/12, 12/02, 02/-12, -12/-22, -22/-32, -32/-42,  
33/23, 23/13, 13/03, 03/-13, -13/-23, -23/-33,  
44/34, 34/24, 24/14, 14/04, 04/-14, -14/-24,  
55/45, 45/35, 35/25, 25/15, 15/05, 05/-15,  
66/56, 56/46, 46/36, 36/26, 26/16, 16/06,
77/67, 67/57, 57/47, 47/37, 37/27, 27/17,
11/12, 01/02, -11/-12, -21/-22, -31/-32, -41/-42,  
22/23, 12/13, 02/03, -12/-13, -22/-23, -32/-33,  
33/34, 23/24, 13/14, 03/04, -13/-14, -23/-24,  
44/45, 34/35, 24/25, 14/15, 04/05, -14/-15,  
55/56, 45/46, 35/36, 25/26, 15/16, 05/06,
66/67, 56/57, 46/47, 36/37, 26/27, 16/17}
        \draw[help lines] (\from) -- (\to);
    \end{tikzpicture}\ytableausetup{boxsize = normal}
  \caption{An example of a growth diagram for \( r=2 \) and \( d = 5 \). We can take the bottom left corner as \( (1,1) \).}
  \label{fig:example-cgd}
\end{figure}
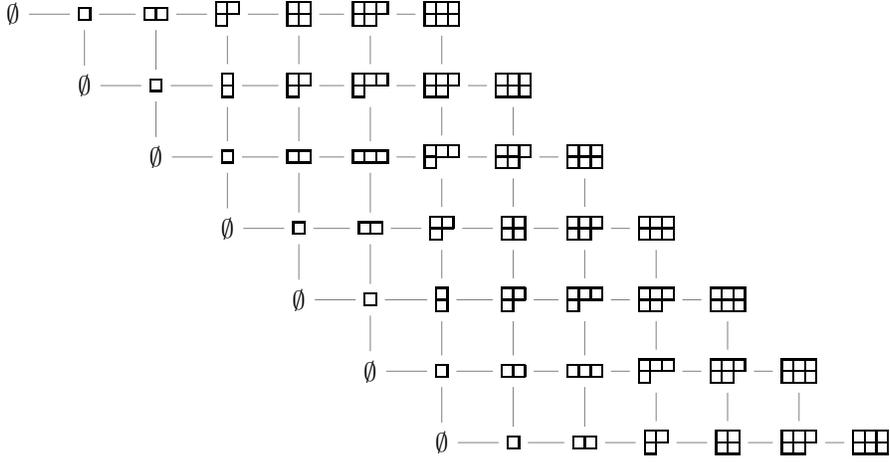

\begin{Remark}
  \label{rem:path-defines-uniquely}
A path though \( \gamma \), as defined in Section~\ref{sec:growth-diagrams} from a node \( (i,i) \) to a node \( (j,j+k) \) (i.e. nodes lying on the left and right edges of the diagram) completely defines all of the partitions lying in the rectangular region the path spans. In the case of cylindrical diagrams the extra condition that \( \gamma_{i(i+k)} \) is the rectangular shape means such a path completely determines the entire cylindrical growth diagram.
\end{Remark}

\begin{Proposition}[{\cite[Lemma~6.4 and Theorem~6.5]{Speyer:2014gg}}]
  \label{prp:fibre-indexed-by-cgd}
For an associahedron \( \associahedron \subset \Ss(\sq^k)(\RR) \) the map \( \gamma \) is a cylindrical growth diagram. The associahedra which tile \( \Ss(\sq^k)(\RR) \) are labelled by pairs \( (s,\gamma) \) of a circular ordering \( s \) and a cylindrical growth diagram \( \gamma \).
\end{Proposition}

Note the cylindrical growth diagram depends on the particular representative of the circular ordering \( s \in S_k / D_k \) that we choose. If we choose another representative the cylindrical growth diagram shifted or we take the mirror image (this comes from the action of \( D_k \)).

\subsubsection{Wall crossing in the fundamental case}
\label{sec:wall-crossing}

We now recall a description of how these associahedra are joined together. Fix an associahedron \( \associahedron \) in \( \Ss(\sq^k)(\RR) \) labelled by \( (s,\gamma) \). Let \( (\hat{s},\hat{\gamma}) \) be the labelling of the associahedra \( \hat{\associahedron} \) joined to \( \associahedron \) by the facet \( \associahedron_{pq} \). Using the description of \( \M{k}(\RR) \) the circular ordering \( \hat{s} \) is obtained from \( s \) by reversing the order of \( s(p),s(p+1),\ldots,s(q-1) \).

\begin{Proposition}[Proposition~6.7 \cite{Speyer:2014gg}]
  \label{prp:crossing-walls}
The cylindrical growth diagram \( \hat{\gamma} \) is given by
\begin{equation}
  \label{eq:new-cgd}
\hat{\gamma}_{ij} =
\begin{cases}
\gamma_{ij} &\text{if } [i,j]\cap [p,q] = \emptyset \text{ or } [p,q] \subseteq [i,j], \\
\gamma_{(p+q-j)(p+q-i)} &\text{if } [i,j] \subseteq [p,q].
\end{cases}
\end{equation}
\end{Proposition}

\begin{proof}
We will not repeat the proof here except to say that since the partitions \( \gamma_{ij} \) are constant along the relevant divisor of \( \M{k}(\RR) \), the partitions do not change, only the indexation relative to the circular ordering changes. Considering how the indexation changes allows one to write the conditions in~(\ref{eq:new-cgd}).
\end{proof}

Note \( \hat{\gamma} \) can be determined recursively by the information given in~(\ref{eq:new-cgd}). The pairs \( (i,j) \) which appear in~(\ref{eq:new-cgd}) are those for which \( \associahedron_{ij} \) intersects \( \associahedron_{pq} \). The rule means when we cross a wall we flip certain triangles and leave others fixed. This is depicted in Figure~\ref{fig:cross-walls-flip}. The red triangle is flipped about the axis shown, green triangles are fixed, and other areas are computed recursively (or using the cyclicity properties of the diagram).

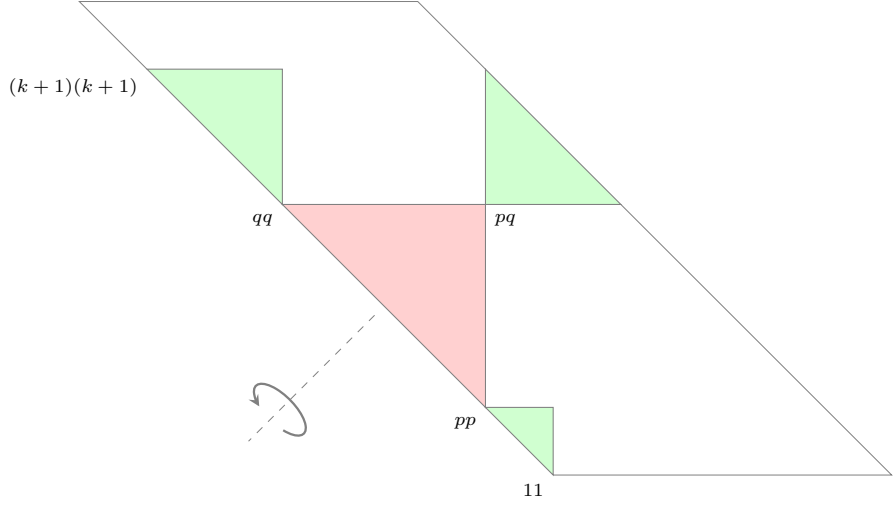
\begin{figure}
  \centering
  \begin{tikzpicture}[scale=0.9]
    \node (ll) at (0,0) {};
    \node (ul) at (-7,7) {};
    \node (lr) at (5,0) {};
    \node (ur) at (-2,7) {};

    \node[anchor=north west] (pq) at (-1,4) {\scriptsize{\( pq \)}};
    \node[anchor=north east] (qq) at (-4,4) {\scriptsize{\( qq \)}};
    \node[anchor=north east] (pp) at (-1,1) {\scriptsize{\( pp \)}};

    \draw[fill=red!30,fill opacity=0.6,draw=gray] (-1,1) -- (-1,4) -- (-4,4) -- cycle;

    \node[anchor=north east] (11) at (0,0) {\scriptsize{\( 11 \)}};
    \node[anchor=north east] (k1k1) at (-6,6) {\scriptsize{\( (k+1)(k+1) \)}};
    \node[anchor=north east] (ut) at (-1,6) {};
    \node[anchor=north east] (lt) at (1,4) {};
    \node[anchor=north east] (qk1) at (-4,6) {};
    \node[anchor=north east] (1p) at (0,1) {};

    \draw[fill=green!30,fill opacity=0.6,draw=gray] (-1,6) -- (-1,4) -- (1,4) -- cycle;
    \draw[fill=green!30,fill opacity=0.6,draw=gray] (-4,4) -- (-4,6) -- (-6,6) -- cycle;
    \draw[fill=green!30,fill opacity=0.6,draw=gray] (-1,1) -- (0,1) -- (0,0) -- cycle;

    \node (mid) at (-2.5,2.5) {};
    \draw[help lines, dashed] (mid) -- (-4.5,0.5) node[near end,solid] {\AxisRotator[rotate=45]};

    \draw[help lines] (0,0) -- (5,0) -- (1,4);
    \draw[help lines] (-1,6) -- (-2,7) -- (-7,7) -- (-6,6);

  \end{tikzpicture}
  \caption{Crossing walls and flipping triangles.}
  \label{fig:cross-walls-flip}
\end{figure}

\subsubsection{Another realisation of $ \Ss(\lamb) $}
\label{sec:another-realisation}

Before we are able to describe the CW-structure for more general \( \lamb \), we realise \( \Ss(\lamb) \) as a subvariety of \( \Ss(\Box^{\tilde{k}}) \), in fact, the real points will be a CW-subcomplex. Here, \( \tilde{k} = \abs{\lamb} \). We should think of obtaining the family \( \Ss(\lamb) \) inside \( \Ss(\Box^{\tilde{k}}) \) by colliding the first \(  \abs{\lambda_1}  \) marked points in such a way to obtain \( \lambda_1 \), the next \( \abs{\lambda_2} \) to obtain \( \lambda_2 \), and so on. 

For our purposes here, we will take \( \M{2} \) to be a single point. With this convention, we have an embedding \( \M{k} \times \prod_{i = 1}^k \M{ \abs{\lambda_i }+ 1} \) into \( \M{\tilde{k}} \) by sending the tuple \( (C,C^{(1)},C^{(2)},\ldots,C^{(k)}) \) to the stable curve obtained by the following process:
\begin{itemize}
\item If \( \abs{\lambda_i} \ge 2 \), glue the last marked point of \( C_i \) and the \( i^{\text{th}} \) marked point of \( C \).
\item The \( l^{\text{th}} \) marked point of \( C_i \) is renumbered \( l + \sum_{j=1}^{i-1} \abs{\lambda_j} \), for \( 1 \le l \le \abs{\lambda_i} \).
\item If \( \abs{\lambda_i} = 1 \) then the \( i^{\text{th}} \) marked point of \( C \) is renumbered \( \sum_{j=1}^i \abs{\lambda_j} \).
\end{itemize}
This is an example of a \emph{clutching map} as described in~[Section~3]\cite{Knudsen:1983ww} where it is also shown that this is in fact a closed embedding. Figure~\ref{fig:generic-point-in-prod} shows generically what such a stable curve looks like.

\begin{figure}[b]
  \centering
  \begin{tikzpicture}
\draw[help lines,postaction={decorate,decoration={
      markings,
      mark=at position   40/360 with {\node[black,circle,fill,inner sep=0,minimum size=2pt,label={[gray]220:$1$},label={[gray] 40:$3$}] (i1) {};},
      mark=at position  120/360 with {\node[black,circle,fill,inner sep=0,minimum size=2pt,label={[gray]300:$2$},label={[gray]120:$5$}] (i2) {};},
      mark=at position  200/360 with {\node[black,circle,fill,inner sep=0,minimum size=2pt,label={[gray] 20:$3$},label={[gray]200:$4$}] (i3) {};},
      mark=at position  320/360 with {\node[black,circle,fill,inner sep=0,minimum size=2pt,label={[gray]140:$4$},label={[gray]320:$3$}] (i4) {};}
  }}]
(0,0) circle (40pt);

\draw[help lines, postaction={decorate,decoration={
      markings,
      mark=at position    0/360 with {\node[black,circle,fill,inner sep=0,minimum size=4pt,label={[black]  0:$2$},label={[gray]180:$2$}] (2) {};},
      mark=at position  130/360 with {\node[black,circle,fill,inner sep=0,minimum size=4pt,label={[black]130:$1$},label={[gray]310:$1$}] (1) {};}
  }}]
(40:65pt) circle (25pt);

\draw[help lines, postaction={decorate,decoration={
      markings,
      mark=at position    0/360 with {\node[black,circle,fill,inner sep=0,minimum size=4pt,label={[black]  0:$4$},label={[gray]180:$2$}] (4) {};},
      mark=at position  100/360 with {\node[black,circle,fill,inner sep=0,minimum size=4pt,label={[black]100:$3$},label={[gray]280:$1$}] (3) {};},
      mark=at position  160/360 with {\node[black,circle,fill,inner sep=0,minimum size=4pt,label={[black]160:$6$},label={[gray]  0:$4$}] (6) {};},
      mark=at position  200/360 with {\node[black,circle,fill,inner sep=0,minimum size=4pt,label={[black]200:$5$},label={[gray]  0:$3$}] (5) {};}
  }}]
(120:65pt) circle (25pt);

\draw[help lines, postaction={decorate,decoration={
      markings,
      mark=at position  100/360 with {\node[black,circle,fill,inner sep=0,minimum size=4pt,label={[black]100:$7$},label={[gray]280:$1$}] (7) {};},
      mark=at position  160/360 with {\node[black,circle,fill,inner sep=0,minimum size=4pt,label={[black]160:$8$},label={[gray]340:$2$}] (8) {};},
      mark=at position  270/360 with {\node[black,circle,fill,inner sep=0,minimum size=4pt,label={[black]270:$9$},label={[gray] 90:$3$}] (9) {};}
  }}]
(200:65pt) circle (25pt);

\draw[help lines, postaction={decorate,decoration={
      markings,
      mark=at position    0/360 with {\node[black,circle,fill,inner sep=0,minimum size=4pt,label={[black]  0:$11$},label={[gray]180:$2$}] (11) {};},
      mark=at position  230/360 with {\node[black,circle,fill,inner sep=0,minimum size=4pt,label={[black]230:$10$},label={[gray] 50:$1$}] (10) {};}
  }}]
(320:65pt) circle (25pt);

\node at (  0: 0pt) {$C$};
\node at ( 40:65pt) {$C_1$};
\node at (120:65pt) {$C_2$};
\node at (200:65pt) {$C_3$};
\node at (320:65pt) {$C_4$};
  \end{tikzpicture}
  \caption{A generic point in \( \M{4} \times \prod_{i = 1}^k \M{\left| \lambda_i \right| + 1} \) when \( \abs{\lambda_2} = 4 \), \( \abs{\lambda_3} = 3 \) and \( \abs{\lambda_1} = \abs{\lambda_4} = 2 \). The original label of each marked point is shown in grey on the inside of each curve.}
  \label{fig:generic-point-in-prod}
\end{figure}
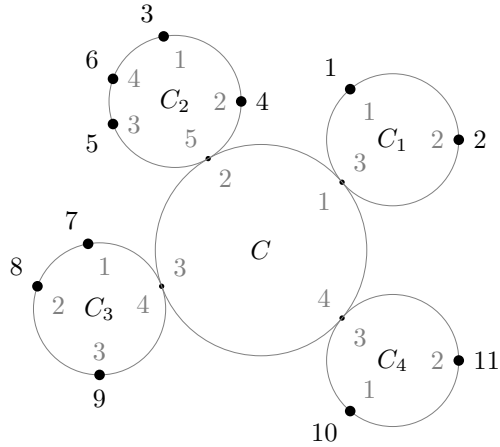

Restrict the family \( \Ss(\Box^{\tilde{k}}) \) to \( \M{k} \times \prod_{i = 1}^k \M{\left| \lambda_i \right| + 1} \) and let \( \mathcal{Y} \) be the connected components where the \( k \) central nodes are labelled by \( \lambda_1,\lambda_2,\ldots,\lambda_k \). As families over \( \M{k} \times \prod_{i = 1}^k \M{\left| \lambda_i \right| + 1} \), \( \mathcal{Y} \) is isomorphic to \( \Ss(\lamb)\times \prod_{i=1}^k \Ss(\Box^{\left| \lambda_i \right|},\lambda_i^\cc) \).

\subsubsection{Dual equivalence cylindrical growth diagrams}
\label{sec:decgd}

We now describe the dual equivalence version of cylindrical growth diagrams. Let \( \lamb \) be a sequence of partitions such that \( \abs{\lamb} = r(d-r) \). It will be convenient here to always take our indices, when referring to this sequence, modulo \( k \). That is, by \( \lambda_l \) we will always mean \( \lambda_{(l \mod k)} \). 

\begin{Definition}
\label{def:decgd}
  A \emph{dual equivalence cylindrical growth diagram} (or decgd for short) of shape \( \lamb \) is a dual equivalence growth diagram \( \gamma \) on \( \II_k \) with \( \gamma_{i(i+k)} = \Lambda \), and such that \( \gamma_{i(i+1)} = \lambda_{i} \). We denote the set of decgd's of shape \( \lamb \) by \( \decgd{\lamb} \).
\end{Definition}

As an example we can take \( r=5 \) and \( d=2 \) again. If we choose 
\begin{equation*}
\ytableausetup{smalltableaux} 
\lamb = \left( \,\ydiagram{2,1}, \,\ydiagram{1}\,\!, \,\ydiagram{2} \right)
\ytableausetup{nosmalltableaux}
\end{equation*}
then Figure~\ref{fig:example-decgd} gives an example of a decgd of shape \( \lamb \). Since shapes with only a single box, as well as shapes of normal or antinormal shape (see Theorem~\ref{thm:props-of-dual-equiv}~(\ref{item:normal-all-de})) only have a single dual equivalence class we do not indicate the dual equivalence class for edges that correspond to such a shape.

\begin{figure}
  \centering
  \ytableausetup{nosmalltableaux}
  \ytableausetup{boxsize=0.4em}
  \begin{tikzpicture}[scale=0.95]
    \node (11)  at (-1,1) {\( \emptyset \)};
    \node (-11) at (1,1)  {\ydiagram{2}};
    \node (-21) at (2,1)  {\ydiagram{2,1}};
    \node (-51) at (5,1)  {\ydiagram{3,3}};
    \node (44)  at (-4,4) {\( \emptyset \)};
    \node (14)  at (-1,4) {\ydiagram{2,1}};
    \node (-14) at (1,4)  {\ydiagram{3,2}};
    \node (-24) at (2,4)  {\ydiagram{3,3}};

    \node (55)  at (-5,5) {\( \emptyset \)};
    \node (45)  at (-4,5) {\ydiagram{1}};
    \node (15)  at (-1,5) {\ydiagram{3,1}};
    \node (-15) at (1,5)  {\ydiagram{3,3}};

    \node (77)  at (-7,7) {\( \emptyset \)};
    \node (57)  at (-5,7) {\ydiagram{2}};
    \node (47)  at (-4,7) {\ydiagram{2,1}};
    \node (17)  at (-1,7) {\ydiagram{3,3}};

    \foreach \from/\to in {
11/-11, -11/-21, -21/-51,
44/14, 14/-14, -14/-24,
55/45, 45/15, 15/-15,
77/57, 57/47, 47/17,
11/14, -11/-14, -21/-24,
44/45, 14/15, -14/-15,
55/57, 45/47, 15/17}
        \draw[help lines] (\from) -- (\to);
\ytableausetup{boxsize = normal}
\ytableausetup{smalltableaux}
\node[fill=white] at (-2.5,5) {\ytableaushort{\none12,3}};
\node[fill=white] at (0,4) {\ytableaushort{\none2,1}};
\node[fill=white] at (1,2.5) {\ytableaushort{\none\none1,23}};
    \end{tikzpicture}\ytableausetup{boxsize = normal}
  \caption{An example of a dual equivalence cylindrical growth diagram for \( r=5 \) and \( d = 2 \).}
  \label{fig:example-decgd}
\end{figure}
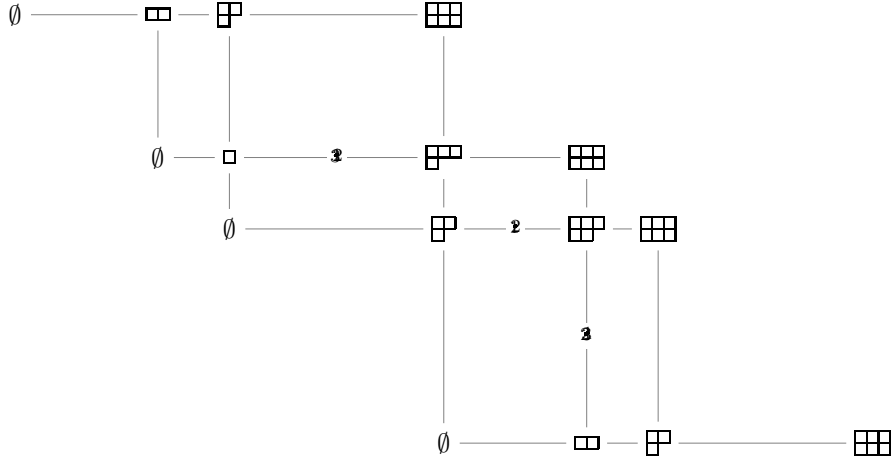

\subsubsection{Reduction of cylindrical growth diagrams}
\label{sec:reduct-cgds}

The reason for the strange choice of layout in Figure~\ref{fig:example-decgd} is the following. If we superimpose the diagram on top of Figure~\ref{fig:example-cgd} we can see that it was simply obtained by forgetting certain nodes in Figure~\ref{fig:example-cgd} but remembering the dual equivalence classes defined by the paths between nodes that we kept. In fact it is the reduction modulo \( m \) of the cylindrical growth diagram from Figure~\ref{fig:example-cgd}.

\begin{Lemma}
  \label{lem:sets-adapted}
Let \( \tilde{k} = r(d-r) \). The set \( \II_{\tilde{k}} \) is adapted to \( m(i) = \abs{\lambda_i} \) and 
\begin{equation*}
  \setc{(i,j) \in \ZZ^2_+}{\bar{m}(i,j) \in \II_{\tilde{k}}} = \II_k.
\end{equation*}
The reduction modulo \( m \) of a cylindrical growth diagram on \( \II_{\tilde{k}} \) for \( (d,r) \) is a decgd on \( \II_k \) for \( (d,r) \) of shape \(  \gamma_{\bar{m}(1,2)}, \gamma_{\bar{m}(2,3)}, \ldots, \gamma_{\bar{m}(k,k+1)} \).
\end{Lemma}

\begin{proof}
The only way for all four corners of a rectangular subset of \( \ZZ^2_+ \) to be contained in \( \II_{\tilde{k}} \) is if all its vertices are in \( \II_{\tilde{k}} \). Hence \( \II_{\tilde{k}} \) is adapted to \( m \).

Suppose that \( (i,j) \in \II_k \), thus \( j-i \le k \). We would like to show \( \bar{m}(i,j) \in \II_{\tilde{k}} \), that is we would like to show \( \hat{m}(j)-\hat{m}(i) = m_s(i,j) \le \tilde{k} \). But
\begin{equation*}
  m_s(i,j) = \sum_{l=i}^{j-1} m(l) = \sum_{l=i}^{j-1} \abs{\lambda_{(l \mod k)}},
\end{equation*}
and since \( j-i \le k \) each \( \lambda_l \) occurs at most once in the above sum. By the assumption that \( \abs{\lamb} = \tilde{k} \) we have \( m_s(i,j) \le \tilde{k} \) as required. Now consider \( (i,j) \in \ZZ^2_+ \) such that \( m_s(i,j) \le \tilde{k} \). If \( j-i > k \) then by the pigeonhole principle each \( \lambda_l \) must occur at least once, and some \( \lambda_l \) must occur twice in the sum
\begin{equation*}
  m_s(i,j) = \sum_{l=i}^{j-1} \abs{\lambda_{(l \mod k)}},
\end{equation*}
contradicting the fact that \( m_s(i,j) \le \tilde{k} \). This proves the second claim.

The only thing left to check to ensure that the reduction modulo \( m \) of a cylindrical growth diagram on \( \II_{\tilde{k}} \) is a decgd on \( \II_k \) is that for any \( i \), \( \bar{m}(i,i+k) = (j,j+\tilde{k}) \) for some \( j \). Equivalently we check that \( m_s(i,i+k) = \tilde{k} \). This is straightforward since the sum
\begin{equation*}
  m_s(i,i+k) = \sum_{l=i}^{i+k-1} \abs{\lambda_{(l \mod k)}}
\end{equation*}
contains one of each \( \lambda_l \) appearing in \( \lamb \) and by assumption \( \abs{\lamb} = \tilde{k} \).
\end{proof}

\begin{Proposition}[Proposition~8.1 in~\cite{Speyer:2014gg}]
  \label{prp:number-of-decgds}
Every decgd on \( \II_k \) is the restriction of a cylindrical growth diagram on \( \II_{\tilde{k}} \). The number of decgd's of shape \( \lambda_{\bullet} \) is \( c^{\Lambda}_{\lamb} \).
\end{Proposition}

\subsubsection{The CW-structure and wall crossing for  general \( \lamb \)}
\label{sec:general-cw-structure}

\begin{Theorem}[{\cite[Theorem~8.2]{Speyer:2014gg}}]
  \label{thm:labelling-of-fibs}
The maximal faces of the CW-structure on \( \Ss(\lamb) \) are labelled by pairs \( (s,\gamma) \) of a circular ordering \( s \) and a decgd \( \gamma \) of shape \( (\lambda_{s(1)},\lambda_{s(2)}\ldots,\lambda_{s(k)}) \).
\end{Theorem}

We leave it to the reader to consult~\cite{Speyer:2014gg} for a rigorous proof of this fact however we will comment on how the results of Section~\ref{sec:another-realisation} allow us to make this statement and produce the dual equivalence classes. We use the notation of Section~\ref{sec:another-realisation}. Let \( \associahedron \) be a \( (k-3) \)-associahedron in \( \Ss(\lamb) \). We consider the embedding
\begin{equation*}
  \begin{tikzcd}
      \associahedron \times \prod_{i=1}^k \Ss(\Box^{\left| \lambda_i \right|},\lambda_i^C)(\RR) \arrow[hook]{r} & \Ss(\Box^{\tilde{k}})(\RR).
  \end{tikzcd}
\end{equation*}
Since this is an embedding of CW-complexes \(  \associahedron \times \prod_{i=1}^k \Ss(\Box^{\left| \lambda_i \right|},\lambda_i^C)(\RR) \) must be contained in some \( (\tilde{k}-3) \)-associahedron \( \tilde{\associahedron} \) of \( \Ss(\Box^{\tilde{k}})(\RR) \). Let \( (\tilde{s},\tilde{\gamma}) \) be the circular order (of \( \tilde{k} \)) and cylindrical growth diagram labelling \( \tilde{\associahedron} \) as per Proposition~\ref{prp:fibre-indexed-by-cgd}. Let \( \tau \) be the unique order preserving bijection
\begin{equation*}
  \setc{\tilde{s}(k^{i}_1)}{1 \le i \le k} \longrightarrow \set{1,2,\ldots,k}, \quad \text{where } k^i_1 = 1+ \sum_{j=1}^{i-1} \abs{\lambda_{j}},
\end{equation*}
then \( s(i) = \tau \circ \tilde{s}(k^{i}_1) \). Let \( m(i) = \abs{\lambda_{s(i)}} \), then \( \gamma \) is the reduction of \( \tilde{\gamma} \) modulo \( m \).

\begin{Proposition}
  \label{prp:crossing-walls-in-general}
Suppose we have two neighbouring associahedra of \( \Ss(\lamb) \) labelled by the pairs \( (s,\gamma) \) and \( (\hat{s},\hat{\gamma}) \) where \( \hat{s} \) is obtained from \( s \) by reversing \( s(p),s(p+1),\ldots,s(q-1) \). The dual equivalence cylindrical growth diagram \( \hat{\gamma} \) is given by
\begin{equation}
  \label{eq:new-decgd}
\hat{\gamma}_{ij} =
\begin{cases}
\gamma_{ij} &\text{if } [i,j]\cap [p,q] = \emptyset \text{ or } [p,q] \subseteq [i,j], \\
\gamma_{(p+q-j)(p+q-i)} &\text{if } [i,j] \subseteq [p,q],
\end{cases}
\end{equation}
with the dual equivalence classes \( \alpha_{ij} \) and \( \beta_{ij} \) being similarly flipped inside the triangle south-west of the node \( (p,q) \). We have the same picture of flipping triangles as in Figure~\ref{fig:cross-walls-flip}.
\end{Proposition}

\subsection{The $ S_k $-action}
\label{sec:Sk-action}

Given a permutation \( \sigma \in S_k \) and a subset \( A \subset [ k ] \) we use the notation
\begin{equation*}
 \sigma A \deq \setc{\sigma(a)}{a \in A}.
\end{equation*}
In this way \( \sigma \) defines a permutation of the set of three element sets \( A \subset [k] \). The permutation \( \sigma \) also induces an isomorphism \( C_A \rightarrow C_{\sigma A} \) (sending marked points to marked points) and hence an isomorphism \( \Gr(d,r)_A \rightarrow \Gr(d,r)_{\sigma A} \). In order to keep our notation tidy we will use \( \sigma \) to denote all of these isomorphisms, the context should make it clear which we a referring to. Since the constructions were functorial, diagrams of the type
\begin{equation}\label{eq:sigma-commutes-phi}
  \begin{tikzcd}[column sep=tiny]
    \Gr(r,d)_A \arrow{rr}{\sigma} && \Gr(r,d)_{\sigma A} \\
    & \Gr(r,d)_{\PP^1} \arrow{ul}{\phi_A} \arrow{ur}[swap]{\phi_{\sigma A}}
  \end{tikzcd}
\end{equation}
commute. Let \( S_k \) act on \( \M{k} \) by permuting marked points. The above discussion means we have an action of \( S_k \) on the trivial family
\begin{equation*}
  \M{k} \times \prod_{A \in \binom{[\, k\,]}{3}} \Gr(d,r)_A.
\end{equation*}

\begin{Proposition}
  \label{prp:Sk-action}
The variety \( \Gg(d,r) \) is stable under the action of \( S_k \) and the variety \( \Ss(\lamb) \) is sent isomorphically onto \( \Ss(\sigma\cdot\lamb) \). In particular the stabiliser of \( \lamb \) in \( S_k \) acts on \( \Ss(\lamb) \).
\end{Proposition}

\begin{proof}
Recall \( \Gg(r,d) \) is defined as the closure of the image of the embedding
\begin{equation*}
  \begin{tikzcd}[row sep = 0pt]
      \oM{k} \times \Gr(r,d) \arrow[hook]{r} & \M{k} \times \prod_A \Gr(r,d)_A \\
      (C,X) \arrow[mapsto]{r} & (C,\phi_A(C,X)).
  \end{tikzcd}
\end{equation*}
By the commutativity of~(\ref{eq:sigma-commutes-phi}), the \( S_k \)-action preserves the image of \( \Gr(r,d)_{\PP^1} \). Thus the action of \( S_k \) also preserves the closure.

Let \( A \subseteq [n] \) be a three element set. By Lemma~\ref{lem:isomorphism-of-schuberts}, for any \( a \in A \) the isomorphism \( \sigma \) sends \( \Omega(\lambda_{a},a)_A \) to the Schubert variety \( \Omega(\lambda_{a},\sigma(a))_{\sigma A} \). This means
\begin{align*}
  \sigma \Ss(\lamb) &= \sigma \left( \Gg(r,d) \cap \bigcap_{a \in A} \Omega(\lambda_{a},a)_A \right) \\
&= \Gg(r,d) \cap \bigcap_{a \in A} \Omega(\lambda_{a},\sigma(t))_{\sigma A} \\
&= \Gg(r,d) \cap \bigcap_{a \in A} \Omega(\lambda_{\sigma^{-1}(a)},a)_A. \qedhere
\end{align*}
\end{proof}

\subsection{Acting on the $ \nu $ labelling}
\label{sec:acting-nu}

We will also need some finer information on exactly what the orbits in \( \Ss(\lamb) \) look like. This will be helpful when it comes to determining what the \( S_k \)-action does to the cylindrical growth diagram indexing a face of \( \Ss(\RR) \).

Consider the fibre \( \Gg(r,d)(C) \) for some stable curve \( C \in \M{k} \). Using the description of the fibre from Theorem~\ref{thm:fibre-theorem}, we have
\begin{equation*}
\Gg(r,d)(C) = \bigcup_{\nu \in \ttN_C} \prod_{i} \bigcap_{d \in D_i} \Omega(\nu(C_i,d),d)_{C_i},
\end{equation*}
where \( \ttN_C \) is the set of node labellings for \( C \), \( C_i \) the irreducible components of \( C \) and \( D_i \) the set of nodes on the component \( C_i \). The action of \( S_k \) on \( \M{k} \) permutes marked points, thus \( C \) and its image \( \sigma C \) are the same curve simply with different marked points. That is, there is an isomorphism \( C \rightarrow \sigma C \) which we can take to be the identity morphism, which sends the point marked by \( a \) to the point marked by \( \sigma(a) \). In this way we identify the irreducible components \( C_i \) and \( \sigma C_i \) and if \( d \) is a node in \( C_i \), we also have a node \( d \in \sigma C_i \). Using this identification, a node labelling \( \nu \in \ttN \) naturally determines a node labelling in \( \ttN_{\sigma C} \), which we also denote by \( \nu \).

\begin{Lemma}
  \label{lem:nu-component}
The \( S_k \) action on \( \Gg(d,r) \) ``preserves the \( \nu \)-component of the fibre''. More precisely, if we fix a \( \nu \in \ttN_C \) the image of
\begin{equation*}
\prod_{i} \bigcap_{d \in D_i} \Omega(\nu(C_i,d),d)_{C_i}
\end{equation*}
under the action of \( \sigma \) is
\begin{equation*}
\prod_{i} \bigcap_{d \in \sigma D_i} \Omega(\nu(\sigma C_i,d),d)_{\sigma C_i}.
\end{equation*}
\end{Lemma}

\begin{proof}
By the commutativity of~(\ref{eq:sigma-commutes-phi}), the Grassmannian \( \Gr(r,d)_{C_i} \) is sent isomorphically onto \( \Gr(r,d)_{\sigma C_i} \). Lemma~\ref{lem:isomorphism-of-schuberts} then tells us that \( \Omega(\nu(C_i,d),d)_{C_i} \) is mapped onto \( \Omega(\nu(\sigma C_i,d),d)_{\sigma C_i} \).
\end{proof}

\subsection{The $ S_k $-action on $ \Ss(\lamb)(\RR) $}
\label{sec:the-action-on-real}

We can now describe how the \( S_k \)-action effects the labelling of the fibre by cylindrical growth diagrams.

\begin{Proposition}
  \label{prp:Action-on-cgds}
If \( \associahedron \) is an associahedron in \( \Ss(\lamb)(\RR) \) labelled by \( (s,\gamma) \) then \( \sigma \associahedron \) is the associahedron labelled by \( (\sigma \cdot s, \gamma) \).
\end{Proposition}

\begin{proof}
Let \( \hat{\associahedron} = \sigma \associahedron \). We first restrict to the fundamental case when \( \lamb = (\sq)^k \). It is clear by the action of \( S_k \) on \( \M{k}(\RR) \) that \( \sigma \cdot s \) is the circular ordering labelling the associahedron \( \hat{\associahedron} \). Recall the cylindrical growth diagram \( \hat{\gamma} \) of \( \hat{\associahedron} \) is determined by considering the \( \nu \)-labelling of a point on each of its facets. As shown in Lemma~\ref{lem:nu-component} this \( \nu \)-labelling is preserved, so \( \hat{\gamma} = \gamma \).

Now consider the general case for arbitrary \( \lamb \). We use the notation from Section~\ref{sec:another-realisation}. Choose a permutation \( \tilde{\sigma} \in S_{\tilde{k}} \) such that 
\begin{align*}
\associahedron &\times \prod_{i=1}^k \Ss(\Box^{\left| \lambda_i \right|},\lambda_i^C)(\RR) \subset \tilde{\associahedron} \\
\intertext{is sent to}
\sigma \cdot\associahedron &\times \prod_{i=1}^k \Ss(\Box^{\left| \lambda_i \right|},\lambda_i^C)(\RR) \subset \tilde{\sigma} \cdot \tilde{\associahedron}.
\end{align*}
Here \( \tilde{\associahedron} \) is a choice of associahedron in \( \Ss(\Box^{\tilde{k}})(\RR) \). If \( \tilde{\gamma} \) is the cylindrical growth  diagram labelling \( \tilde{\associahedron} \), then the above shows \( \tilde{\gamma} \) is also the cylindrical growth diagram labelling \( \tilde{\sigma}\cdot \gamma \). But the decgd labelling \( \sigma \cdot \associahedron \) is by definition the reduction of this cylindrical growth diagram, which by assumption is \( \gamma \).
\end{proof}

\subsubsection{The equivariant monodromy}
\label{sec:equiv-monodr}

Let \( k = n+1 \), fix a basepoint \( C \in \oM{n+1}(\RR) \)  and consider the sequence of partitions \( (\lamb,\mu^\cc) \) where \( \lamb = (\lambda_1,\lambda_2,\ldots,\lambda_n) \) and \( \abs{\mu} = \abs{\lamb} \) (this condition implies \( \abs{\lamb} + \abs{\mu} = r(d-r) \)). Let \( S^{\lamb}_n \subseteq S_n \) be the subgroup fixing \( \lamb \). Proposition~\ref{prp:Sk-action} says we have an action of \( S_n \) on the disjoint union
\begin{equation*}
  \bigsqcup_{\sigma} \Ss(\sigma \cdot \lamb, \mu^\cc)(\RR),
\end{equation*}
where \( \sigma \) ranges over a set of representatives for the cosets \( S_n / S^{\lamb}_n \). The cactus group \( J_n \) acts on this family by equivariant monodromy.

Using Theorem~\ref{thm:labelling-of-fibs}, identify the fibre over \( C \) with \( \bigsqcup_\sigma\decgd{\sigma\cdot\lamb,\mu^\cc} \). We could potentially do this in a number of ways but fix one by choosing, for the associahedron containing \( C \), a representation \( s = (s(1),s(2),\ldots,s(n+1)) \) of the corresponding circular ordering. Let \( \gamma \) be the decgd labelling an associahedron lying over \( C \) and let \( \hat{\gamma} \) be the decgd labelling the associahedron obtained from \( \gamma \) by crossing the wall corresponding to flipping the marked points \( s(p), s(p+1), \ldots, s(q) \).

\begin{Corollary}
  \label{cor:monodromy-action}
The equivariant monodromy action of \( J_{n} \) on \( \bigsqcup_\sigma\decgd{\sigma\cdot\lamb,\mu^\cc} \) is given by \( s_{pq}\cdot\gamma = \hat{\gamma} \).
\end{Corollary}

\begin{proof}
Recall from Section~\ref{sec:fundamental-group}, \( s_{pq} \) acts by monodromy around the equivariant loop \( (\alpha,\hat{s}_{pq}) \) where \( \alpha \) is a path from \( C \) to \( \hat{s}_{pq} \cdot C \) passing through the wall which swaps the marked points \( s(p),s(p+1),\ldots,s(q) \). We lift \( \alpha \) to \( \tilde{\alpha} \), the unique path in the covering space \( \bigsqcup_{\sigma} \Ss(\sigma \cdot \lamb, \mu^\cc)(\RR) \) starting at the point over \( C \) labelled \( \gamma \).
By Proposition~\ref{prp:crossing-walls-in-general} the point over \( \hat{s}_{1q} \cdot C \) at the end of \( \tilde{\alpha} \) is labelled \( \hat{\gamma} \). Now Proposition~\ref{prp:Action-on-cgds} says acting by \( \hat{s}_{1q} \) does not change the decgd. Hence \( s_{1q}\cdot\gamma = \hat{\gamma} \).
\end{proof}

\subsubsection{The fundamental case for \( \mu \)}
\label{sec:fundamental-case}

We now restrict to the case when \( \lamb = (\sq^n) \) which we call the \emph{fundamental case for} \( \mu \). In this case \( S^{\lamb}_n = S_n \) so the fibre over \( C \) is identified with \( \decgd{\sq^n,\mu^\cc} \). For \( n = 5 \) a decgd of shape \( (\sq^n,\mu^\cc) \) will have the form shown in Figure~\ref{fig:decgd-shape-lambmu}. Note the partition in position \( (1,6) \) is \( \mu \) (the bottom left corner is in position \( (1,1) \)). This is demonstrated by the following lemma.

\begin{Lemma}
  \label{lem:mu-appears-as-complement}
If \( \gamma \in \decgd{\sq^n,\mu^\cc} \) then \( \gamma_{1(n+1)} = \mu \).
\end{Lemma}

\begin{proof}
By definition \( \gamma_{n(n+1)} = \mu^\cc \). Consider the rectangular subdiagram with corners \( (n+1,n+1), (n+1,n+2), (1,n+1) \) and \( (1,n+2) \). Extend this rectangular region to a growth diagram and let \( \gamma_{1(n+1)} = \nu \),
\begin{equation*}
  \begin{tikzpicture}[scale=1]
    \node (00) at (0,0) {\( \emptyset \)};
    \node (01) at (0,1) {\( \mu^\cc \)};
    \node (50) at (5,0) {\( \nu \)};
    \node (51) at (5,1) {\( \Lambda \)};
    \foreach \from/\to in {00/01,00/50,01/51,50/51} 
        \draw[help lines] (\from) -- (\to); 
    \draw[decorate,decoration={brace,amplitude=5pt}] (-0.2,0) -- (-0.2,1) node[midway,xshift=-12pt] {\( S \)};
    \draw[decorate,decoration={brace,amplitude=5pt,mirror}] (5.2,0) -- (5.2,1) node[midway,xshift=12pt] {\( T \)};
  \end{tikzpicture}
\end{equation*}
Choose tableaux \( S \) and \( T \) lifting the dual equivalence classes as shown. Since our rectangle is a growth diagram, \( S \) must be the rectification of \( T \). However the rectification of a tableaux of shape \( \Lambda \bs \mu \) has shape \( \mu^\cc \). One way to see this is to note \( c^{\Lambda}_{\nu \mu^\cc} \) is the number of \( T \in \SYT(\Lambda \bs \nu) \) slide equivalent to \( S \), since we have produced such a \( T \), \( c^{\Lambda}_{\nu\mu^\cc} > 0 \). But \( c^{\Lambda}_{\nu \mu^\cc} = \delta_{\nu \mu} \) (see \cite[Section~9.4]{Fulton:1997vaa}) hence \( \nu = \mu \).
\end{proof}

\begin{figure}
  \centering
    \ytableausetup{nosmalltableaux}
  \ytableausetup{boxsize=0.4em}
  \begin{tikzpicture}[scale=0.85]
    \node (11)  at (-1,1) {\( \emptyset \)};
    \node (01)  at (0,1)  {\ydiagram{1}};
    \node (-11) at (1,1)  {\( \cdot \)};
    \node (-21) at (2,1)  {\( \cdot \)};
    \node (-31) at (3,1)  {\( \cdot \)};
    \node (-41) at (4,1)  {\( \cdot \)};
    \node (-51) at (5,1)  {\ydiagram{3,3}};

    \node (22)  at (-2,2) {\( \emptyset \)};
    \node (12)  at (-1,2) {\ydiagram{1}};
    \node (02)  at (0,2)  {\( \cdot \)};
    \node (-12) at (1,2)  {\( \cdot \)};
    \node (-22) at (2,2)  {\( \cdot \)};
    \node (-32) at (3,2)  {\( \cdot \)};
    \node (-42) at (4,2)  {\ydiagram{3,3}};

    \node (33)  at (-3,3) {\( \emptyset \)};
    \node (23)  at (-2,3) {\ydiagram{1}};
    \node (13)  at (-1,3) {\( \mu|_{2} \)};
    \node (03)  at (0,3)  {\( \cdot \)};
    \node (-13) at (1,3)  {\( \cdot \)};
    \node (-23) at (2,3)  {\( \cdot \)};
    \node (-33) at (3,3)  {\ydiagram{3,3}};

    \node (44)  at (-4,4) {\( \emptyset \)};
    \node (34)  at (-3,4) {\ydiagram{1}};
    \node (24)  at (-2,4) {\( \cdot \)};
    \node (14)  at (-1,4) {\( \mu|_{3} \)};
    \node (04)  at (0,4)  {\( \cdot \)};
    \node (-14) at (1,4)  {\( \cdot \)};
    \node (-24) at (2,4)  {\ydiagram{3,3}};

    \node (55)  at (-5,5) {\( \emptyset \)};
    \node (45)  at (-4,5) {\ydiagram{1}};
    \node (35)  at (-3,5) {\( \cdot \)};
    \node (25)  at (-2,5) {\( \cdot \)};
    \node (15)  at (-1,5) {\( \mu|_{4} \)};
    \node (05)  at (0,5)  {\( \cdot \)};
    \node (-15) at (1,5)  {\ydiagram{3,3}};

    \node (66)  at (-6,6) {\( \emptyset \)};
    \node (56)  at (-5,6) {\ydiagram{1}};
    \node (46)  at (-4,6) {\( \cdot \)};
    \node (36)  at (-3,6) {\( \cdot \)};
    \node (26)  at (-2,6) {\( \cdot \)};
    \node (16)  at (-1,6) {\( \mu \)};
    \node (06)  at (0,6)  {\ydiagram{3,3}};

    \node (77)  at (-7,7) {\( \emptyset \)};
    \node (67)  at (-6,7) {\( \mu^\cc \)};
    \node (57)  at (-5,7) {\( \cdot \)};
    \node (47)  at (-4,7) {\( \cdot \)};
    \node (37)  at (-3,7) {\( \cdot \)};
    \node (27)  at (-2,7) {\( \cdot \)};
    \node (17)  at (-1,7) {\ydiagram{3,3}};

    \foreach \from/\to in {
11/01, 01/-11, -11/-21, -21/-31, -31/-41, -41/-51,  
22/12, 12/02, 02/-12, -12/-22, -22/-32, -32/-42,  
33/23, 23/13, 13/03, 03/-13, -13/-23, -23/-33,  
44/34, 34/24, 24/14, 14/04, 04/-14, -14/-24,  
55/45, 45/35, 35/25, 25/15, 15/05, 05/-15,  
66/56, 56/46, 46/36, 36/26, 26/16, 16/06,
77/67, 67/57, 57/47, 47/37, 37/27, 27/17,
11/12, 01/02, -11/-12, -21/-22, -31/-32, -41/-42,  
22/23, 12/13, 02/03, -12/-13, -22/-23, -32/-33,  
33/34, 23/24, 13/14, 03/04, -13/-14, -23/-24,  
44/45, 34/35, 24/25, 14/15, 04/05, -14/-15,  
55/56, 45/46, 35/36, 25/26, 15/16, 05/06,
66/67, 56/57, 46/47, 36/37, 26/27, 16/17}
        \draw[help lines] (\from) -- (\to);
    \end{tikzpicture}\ytableausetup{boxsize = normal}
  \caption{A decgd of shape \( (\sq^5,\mu^\cc) \)}
  \label{fig:decgd-shape-lambmu}
\end{figure}

We can use Lemma~\ref{lem:mu-appears-as-complement} to give a bijection between \( \decgd{\sq^n,\mu^\cc} \) and \( \SYT(\mu) \) by choosing a path though \( \II \). The standard Young tableaux associated to a decgd \( \gamma \) is described by the growth diagram along the path. Fix the unique path from \( (1,1) \) to \( (1,n+2) \) and denote it \( \alpha \). Use \( \alpha \) to identify \( \decgd{\sq^n,\mu^\cc} \) with \( \SYT(\mu) \).

There is an action of the cactus group \( J_n \) on \( \SYT(\mu) \) by \emph{partial Sch\"utzenberger involutions}. This action was studied by Berenstein and Kirillov~\cite{Kirillov:1995te}. The partial Sch\"utzenberger involution of order \( q \) on \( T \in \SYT(\mu) \) is defined by applying the Sch\"utzenberger involution to the subtableau \( T|_q \) and leaving the remaining entries (i.e. those in \( T|_{q+1,n} \)) unchanged.

\begin{Proposition}
  \label{prp:flipping-decg-is-schutz}
The identification of \( \SYT(\mu) \) and \( \decgd{\sq^n,\mu^\cc} \) above identifies the action of \( J_n \) on both sides. More precisely if \( T \) is obtained using the path \( \alpha \) from the decgd \( \gamma \) then the standard tableaux \( s \cdot T \) is obtained by taking the path \( \alpha \) through the decgd \( s \cdot \gamma \), for all \( s \in J_n \).
\end{Proposition}

\begin{proof}
We only need to show this for \( s = s_{1q} \) by Lemma~\ref{lem:generators-of-J_n}. Let \( \gamma \) be the decgd with \( T \) along the path \( \alpha \). Denote the shape of \( T|_q \) by \( \mu|_q \), so \( \gamma_{1(q+1)} = \mu|_q \). Consider the triangle in \( \gamma \) depicted in Figure~\ref{fig:action-s_1q}. Proposition~\ref{prp:crossing-walls-in-general} says that \( s_{1q}\cdot \gamma \) will contain the same triangle, flipped about the axis shown. In particular the tableau obtained along the path \( \alpha \) in \( s_{1q} \cdot \gamma \) is the same as the tableau obtained along the path \( \beta \) from \( (q+1,q+1) \) to \( (1,q+1) \) and then to \( (1,n+2) \) in \( \gamma \). By Corollary~\ref{cor:growth-schutz} this is the partial Sch\"utzenberger involution \( s_{1q} \cdot T \).
\end{proof}

\begin{figure}
  \centering
  \ytableausetup{nosmalltableaux}
  \ytableausetup{boxsize=0.4em}
  \begin{tikzpicture}[scale=0.85]
    \node (12)  at (-1,2) {\( \cdot \)};
    \node (-52) at (5,2)  {\( \cdot \)};
    \node (77)  at (-7,7) {\( \cdot \)};
    \node (17)  at (-1,7) {\( \cdot \)};

    \node (15)  at (-1,5) {\( \mu|_{q} \)};

    \node (54)  at (-4.5,5)  {}; 
    \node (16)  at (-1,6) {\( \mu \)};
    \node (06)  at (0.2,6)  {};

    \foreach \from/\to in {
12/-52, 12/77, 77/17, 17/-52, 
12/15, 54/15,
16/17, 16/06}
        \draw[help lines] (\from) -- (\to);

\draw[dashed] (15) -- (-4,2) node[near end,solid] {\AxisRotator[rotate=45]};
\node[anchor=south] at (-4,3) {\( s_{1q} \)};

\draw[dotted, decoration={markings,mark=at position 1 with {\arrow[gray]{>}}},
         postaction={decorate}] (-4.65,5.15) -- (-1.15,5.15) node[midway,anchor=south] {\( \beta \)} -- (-1,5.3) -- (-1,5.85);
    \end{tikzpicture}\ytableausetup{boxsize = normal}
  \caption{The action of \( s_{1q} \)}
  \label{fig:action-s_1q}
\end{figure}
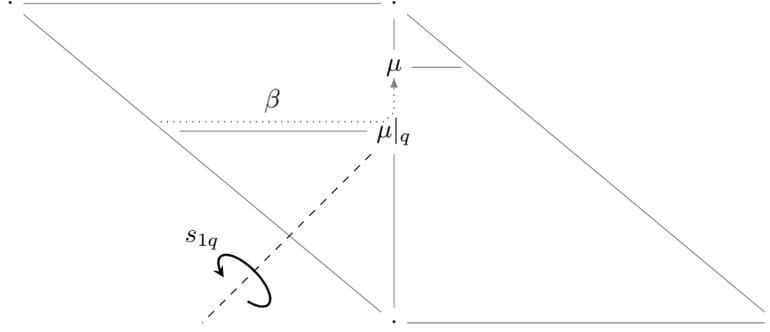

\begin{Corollary}
  \label{cor:eq-monodrom-fund-case}
Identify the fibre in \( \Ss(\sq^n,\mu^\cc)(\RR) \) over \( C \) with \( \SYT(\mu) \) as describe above. The equivariant monodromy action of \( J_n \) is given by partial Sch\"utzenberger involutions.
\end{Corollary}

\section{Bethe algebras}
\label{sec:bethe-algebras}

In this section we define the Bethe algebras and recall the relationship between Bethe algebras and Schubert intersections. We review the notion of Galois theory for a finite morphism of varieties. We then use the results of Section~\ref{sec:speyers-flat-family} prove Theorem~\ref{thm:main-thm}.

\subsection{Definition of Bethe algebras}
\label{sec:def-bethe-algebras}

Let \( \glr[t] := \mathfrak{gl}_r \otimes \CC[t] \) be the \emph{current algebra} of polynomials with coefficients in \( \mathfrak{gl}_r \). For a formal variable \( u \) and an element \( x \in \mathfrak{gl}_r \) define the generating function
\begin{equation*}
  x(u) \deq \sum_{s=0}^\infty (x \otimes t^s) u^{-s-1}.
\end{equation*}
This is a useful accounting device. For example, it allows us to define, for any \( a \in \CC \), an automorphism \( \rho_a \) of \( \mathfrak{gl}_r[t] \) by the assignment, \( x(u) \mapsto x(u-a) \), for every element \( x \in \glr \). This means we map \( x t^s \) to \( x(a+t)^s \), the coefficient of \( u^{-s-1} \) in the expansion of \( x(u-a) \) about infinity. For a \( \glr[t] \)-module \( M \) and a complex number \( a \in \CC \), we define the \emph{evaluation module} \( M(a) \) as the pullback over the map \( \rho_a \).

In a similar fashion we define the \emph{evaluation} morphism \( \mathrm{ev} \map{\glr[t]}{\glr} \) by the assignment \( x(u) \mapsto xu^{-1} \), meaning we send \( t \) to zero and \( x \) to itself. Then any \( \glr \)-module can be made into a \( \glr[t] \)-module by pullback. Given a \( \glr \)-module \( N \), as a \( \glr[t] \)-module, \( t \) acts by zero. Hence on \( N(a) \), \( xt^s \) acts by \( a^sx \).

If \( \partial \) is differentiation with respect to \( u \) then we can define the following noncommutative determinant by expansion along the first column,
\begin{equation*}
  \bdiffop := \det
  \begin{pmatrix}
    \partial - e_{11}(u) &          - e_{21}(u) & \cdots &          - e_{r1}(u) \\
             - e_{12}(u) & \partial - e_{22}(u) & \cdots &          - e_{r2}(u) \\
             \vdots     &      \vdots         & \ddots &       \vdots        \\
             - e_{1r}(u) &          - e_{2r}(u) & \cdots & \partial - e_{rr}(u) \\
  \end{pmatrix}
\end{equation*}
where \( e_{ij} \) are the standard generators for \( \mathfrak{gl}_r \). The determinant \( \bdiffop \) has the form
\begin{equation*}
  \bdiffop = \partial^r + \sum_{i=1}^r B_i(u) \partial^{r-i},
\end{equation*}
for some power series with coefficients \( B_{is} \in \glr[t] \),
\begin{equation*}
  B_i(u) = \sum_{s=i}^\infty B_{is}u^{-s}.
\end{equation*}

\begin{Definition}
  \label{def:bethe-algebra}
The \emph{universal Bethe algebra} is the subalgebra \( \bethe \), of \( U(\glr[t]) \) generated by the coefficients \( B_{is} \). For an \( \bethe \)-module \( M \), we call the image of \( \bethe \) in \( \End(M) \) the \emph{Bethe algebra associated to} \( M \).
\end{Definition}

By~\cite[Propositions~8.2 and~8.3]{Mukhin:2006vt} the universal Bethe algebra is a commutative subalgebra of \( U(\glr[t]) \) and commutes with the action of \( \glr \subset \glr[t] \). As a result, for any \( \glr[t] \)-module \( M \), and any weight \( \lambda \), the subspaces \( M_{\lambda} \), \( M^{\mathsf{sing}} \) and \( M_{\lambda}^{\mathsf{sing}} \subset M \) are \( \bethe \)-submodules. Let \( \lamb \) be a sequence of partitions with at most \( r \) rows. As a special case of Definition~\ref{def:bethe-algebra}, for \( z = (z_1, z_2, \ldots, z_n) \in X_n \) we denote the Bethe algebra associated to
\begin{equation*}
L(\lamb;z)_\mu = {[L(\lambda_{1})(z_1) \otimes L(\lambda_{2})(z_2) \otimes \cdots \otimes L(\lambda_{n})(z_n)]}_{\mu}^{\text{sing}}
\end{equation*}
by \( \bethe(\lamb;z)_\mu \). 

\begin{Lemma}
\label{lem:bethe-inv-Aff-sym}
The Bethe algebras \( \bethe(\lamb;z)_\mu \) are invariant under the action of the group \( \Aff_1 \): if \( \alpha \in \CC^\times, \beta \in \CC \), then \( \bethe(\lamb;\alpha z + \beta)_\mu = \bethe(\lamb;z)_\mu \) as subalgebras of \( \End(\Llambmu) \).
\end{Lemma}

\begin{proof}
This is proved for example in~\cite[Proposition~1]{Rybnikov:2014wh}.
\end{proof}

By Lemma~\ref{lem:bethe-inv-Aff-sym} the Bethe algebras form a family of algebras over \(  \oM{n+1}(\CC) \). Denote the spectrum of this family by \( \pi\map{\bspec(\lamb)_\mu}{\oM{n+1}(\CC)} \) and the fibre over a point \( z \) by \( \bspec(\lamb;z)_\mu \).

\begin{Theorem}[\cite{Mukhin:2009et}, Corollary~6.3]
\label{thm:simplicity}
Suppose \( z_1, z_2, \ldots, z_n \) are distinct real numbers and \( \mu \) is a partition of \( n \) (with at most \( r \) rows). The Bethe algebra \( \bethe(\lamb;z)_\mu \) has simple spectrum. In particular \( \bethe(\lamb;z)_\mu \) has dimension \( c^\mu_{\lamb} \) and over \( \oM{n+1}(\RR) \), \( \pi \) is a covering of degree \( c^{\mu}_{\lamb} \).
\end{Theorem}

\subsection{The MTV isomorphism}
\label{sec:mtv-isomorphism}

In this section we recall the definition of the MTV isomorphism. Choose \( d \ge r \) such that \( r(d-r) \ge \abs{\lamb} \). Mukhin, Tarasov and Varchenko~\cite{Mukhin:2009et} define the \emph{MTV-isomorphism} \( \theta\map{\bspec(\lamb)_\mu}{\Omega(\lamb,\mu^\cc)} \) in the following way. Let \( \chi \in \bspec(\lamb;z)_\mu \), be an element of the fibre over \( z \). We consider \( \chi \) as a map \( \chi\map{\bethe(\lamb;z)_\mu}{\CC} \). Let \( b_{is} = \chi(B_{is}) \) be the image of the generators. Define
\begin{equation*}
  b_i(u) = \sum_{s=i}^\infty b_{is}u^{-s}.
\end{equation*}
Consider the differential operator
\begin{equation*}
  \bdiffop^\chi = \partial^r + \sum_{i=1}^r b_i(u)\partial^{r-i},
\end{equation*}
which is just the differential operator \( \bdiffop \) evaluated on the eigenspace of \( \Llambmu \) corresponding to \( \chi \). Define \( \theta(\chi) \) to be the kernel of \( \bdiffop^\chi \) acting on \( \CC_d[u] \).

\begin{Theorem}[{\cite{Mukhin:2009et}}]
  \label{thm:mtv-iso}
The subspace \( \theta(\chi) \subseteq \CC_d[u] \) has dimension \( r \) and is contained in the Schubert intersection \( \Omega(\lamb,\mu^\cc;z,\infty) \). Moreover it defines an isomorphism of families over \( \oM{n+1}(\CC) \) of the varieties \( \bspec(\lamb)_\mu \) and \( \Omega(\lamb,\mu^\cc) \).
\end{Theorem}

\subsection{Crystals}
\label{sec:crystals}

We recall briefly how crystals of the irreducible \( \glr \)-modules are realised using semistandard tableaux. For a complete description see~\cite{Hong:2002vh}. The crystal \( \crys = \crys(\sq) \) of the vector representation is 
\begin{equation*}
  \begin{tikzcd}
\ytableausetup{centertableaux}\begin{ytableau}1\end{ytableau} \arrow{r}{1} & \ytableausetup{centertableaux}\begin{ytableau}2\end{ytableau} \arrow{r}{2} & \cdots \arrow{r}{r-1} & \ytableausetup{centertableaux}\begin{ytableau}r\end{ytableau}.
  \end{tikzcd}
\end{equation*}
Using the tensor rule we have a description of the crystal  \( \crys^{\ox n} \) as the set \( \words(n) \) of words of length \( n \) in the letters \( 1,2,\ldots,r \). Identify the element \( \yt{{i_1}} \ox \yt{{i_2}} \ox \ldots \ox \yt{{i_n}} \in B^{\ox n} \)  with the word \( i_1i_2\cdots i_n \). The irreducible \( L(\lambda) \) embeds into the tensor power of vector representations \( V^{\ox n} \) where \( n = \abs{\lambda} \). We can thus realise the crystal \( \crys(\lambda) \) as an appropriate connected component of \( \crys^{\ox n} \). The \emph{RSK-correspondence} gives a bijection between words of length \( n \) and pairs of tableaux:
\begin{equation*}
\RSK\map{\words(n)}{\bigsqcup_{\left| \lambda \right| = n} \SSYT(\lambda) \times \SYT(\lambda)}.
\end{equation*}
See~\cite{Fulton:1997vaa} for a definition. We use \( \mathtt{P} \) and \( \mathtt{Q} \) to denote composition of \( \RSK \) with projection onto the first and second factors (the P and Q-symbols of the word). By a Theorem of Ariki and Kazhdan-Lusztig (see~\cite[Theorem~A]{Ariki:2000uw}) \( u,v \in B^{\ox n} \) lie in the same irreducible crystal (on the same connected component) if and only if \( \QRSK(u) = \QRSK(v) \). For a standard \( \lambda \)-tableau \( T \). We can use \( \RSK^{-1}(T,\cdot) \) to embed \( \SSYT(\lambda) \) into \( \crys^{\ox n} \), this is a connected component of the crystal isomorphic to \( \crys(\lambda) \). This does not depend on \( T \) and identifies the vertices of \( \crys(\lambda) \) with semistandard tableaux.

\subsubsection{Coboundary structure}
\label{sec:coboundary-structure}

The category of (finite dimensional) \( \glr \)-modules is a \emph{braided monoidal category}. In~\cite{Henriques:2006is}, Henriques-Kamnitzer show this structure does not descend to the category of crystals and in fact the category of crystals cannot be given a braiding. Instead, Henriques-Kamnitzer show that the category of crystals satisfies the axioms of a \emph{coboundary category}.

A coboundary monoidal category is a monoidal category with a \emph{commutor}, natural isomorphisms \( \sigma_{XY}\map{X \ox Y}{Y \ox X} \), satisfying a certain coherence condition (similar to the hexagon condition for braided monoidal categories). The important fact for us is that the cactus group, \( J_n \), acts on \( n \)-fold tensor products in a coboundary category. For objects \( B_1, B_2, \ldots, B_n \), for \( 1 \le p < q \le n \) set 
\begin{equation*}
\sigma_{pq} = \id^{\ox (p-1)} \ox \sigma_{B_p \ox \cdots \ox B_{q-1},B_q} \ox \id^{\ox (n-q)}.
\end{equation*}
The generators \( s_{pq} \in J_n \) of the cactus group act in the following way. First set \( s_{p(p+1)} = \sigma_{p(p+1)} \) and inductively \( s_{pq} = s_{(p+1)q} \circ \sigma_{pq} \). We should think of \( s_{pq} \) as swapping the order of \( B_p, B_{p+1}, \ldots, B_q \).

In the case \( \mathfrak{g} = \glr \) Henriques-Kamnitzer give a simple description of the commutor for the category of crystals. The Sch\"utzenberger involution \( \xi \) acts on a crystal by acting on each irreducible component (and thus has vertices identified with semistandard tableaux) individually. For two \( \glr \)-crystals, \( B,C \) define 
\begin{equation}
  \label{eq:cactus-commutor}
\sigma_{B,C}\map{B \ox C}{C \ox B}, \quad \text{by} \quad \sigma_{B,C}(b \ox c) = \xi(\xi(c) \ox \xi(b)).
\end{equation}
This defines the cactus commutor on the category of \( \glr \)-crystals. The generator \( s_{pq} \) of the cactus group acts on \( b_1 \ox b_2 \ox \cdots \ox b_n \in B_1\ox B_2\ox \ldots \ox B_n \) by sending it to
\begin{equation*}
b_1 \ox \cdots \ox \xi(\xi(b_q) \ox \xi(b_{q-1}) \ox \cdots \ox \xi(b_p)) \ox \cdots \ox b_n.
\end{equation*}

\begin{Remark}
  \label{rem:schutz-on-words}
As remarked above we can identify \( \words(n) \) with \( \crys^{\ox n} \). To calculate the action of \( \xi \) on \( w \in \crys^{\ox} \) we use the RSK-correspondence. If \( \RSK(w) = (P,Q) \) then \( \xi(w) = \RSK^{-1}(\xi P, Q) \). In particular \( w \) and \( \xi(w) \) have the same Q-symbol.
\end{Remark}

\subsubsection{Crystals and decgds}
\label{sec:crystals-decgds}

Since \( J_n \) acts on crystals by crystal morphisms it preserves weight spaces and singular vectors. Thus \( s \in J_n \) produces a map of sets \( s\map{\Blambmu}{\crys(\hat{s}\cdot\lamb)^{\mathsf{sing}}_\mu} \), recall that \( \hat{s} \) is the image of \( s \) in \( S_n \). Let \( J^{\lamb}_n \subseteq J_n \) be the preimage of \( S^{\lamb}_n \). The group \( J^{\lamb}_n \) acts on the set \( \Blambmu \).

We can also characterise \( J^{\lamb}_n \) as the equivariant fundamental group of the \( S^{\lamb}_n \)-action on \( \M{n+1}(\RR) \). Thus \( J^{\lamb}_n \) also acts on \( \decgd{\lamb,\mu^\cc} \) as described by Corollary~\ref{cor:monodromy-action}. We will defer the proof of the following theorem to Section~\ref{sec:general-case}. 

\begin{Theorem}
  \label{thm:equivariant-map}
There is a \( J_n^{\lamb} \)-equivariant bijection \( \decgd{\lamb, \mu^\cc} \longrightarrow \Blambmu \).
\end{Theorem}

\subsection{Galois actions for finite maps}
\label{sec:galois-actions}

In this section we recall the notion of the Galois group for a finite morphism between varieties. This was defined by Harris in~\cite{Harris:1979to}. Let \( \pi\map{Y}{X} \) be a dominant morphism between varieties over \( \CC \) of equal dimension. We say it has \emph{degree} \( d \) if the associated field extension \( K(X) \hookrightarrow K(Y) \) has degree \( d \). For a generic point \( x \in X \) the fibre consists of \( d \) reduced points, which we denote \( y_1,y_2,\ldots,y_d \). By the primitive element theorem there exists \( \alpha \in K(Y) \) such that \( K(Y) = K(X)[\alpha] \). Let \( P \in K(X)[t] \) be the minimal polynomial of \( \alpha \). By definition, \( P \) has degree \( d \).

Let \( \Mm_x \) be the field of germs of meromorphic functions around \( x \) and \( \Mm_i \) the field of meromorphic functions around \( y_i \). We have natural inclusions \( K(X) \subseteq \Mm_x \) and \( K(Y) \subseteq \Mm_i \). Since \( \pi \) is locally around \( y_i \) an isomorphism of analytic varieties we have isomorphisms \( \phi_i\map{\Mm_i}{\Mm_x} \). Let \( K(Y)_i = \phi_i(K(Y)) \), and let \( L \) be the subfield of \( \Mm_x \) generated by the \( K(Y)_i \). The images \( \alpha_i = \phi_i(\alpha) \in L \) are all distinct and thus are a complete set of roots for \( P \). The field \( L \) is the Galois closure, in \( \Mm_x \) of the extension \( K(X) \hookrightarrow K(Y) \) and thus the Galois group \( \Gal(L/K(X)) \) acts on the set of roots \( \set{\alpha_i} \) which we may identify canonically with the fibre \( \pi^{-1}(x) \).

\begin{Definition}
  \label{def:galois-grp}
The image of \( \Gal(L/K(X)) \) in \( S_{\pi^{-1}(x)} \), the group of permutations of the fibre, is called the \emph{Galois group of} \( \pi \) and is denoted \( \Gal(\pi) \) or \( \Gal(\pi;x) \) if we wish to emphasise the basepoint.
\end{Definition}

\begin{Remark}
  \label{rem:galois-birational-invariant}
The definition of the Galois group \( \Gal(\pi;x) \) depends only on local properties of the morphism \( \pi \), it is thus a birational invariant of \( \pi \). Let \( \pi'\map{X'}{Y'} \) be another degree \( d \), dominant morphism and suppose we have birational maps making the following diagram commute,
\begin{equation*}
  \begin{tikzcd}
    Y \arrow{d}{\pi} \arrow[dashed]{r}{g} & Y' \arrow{d}{\pi'} \\
    X \arrow[dashed]{r}{f} & X'.
  \end{tikzcd}
\end{equation*}
Suppose \( f \) is defined on \( x \) and \( g \) is defined on \( y_1,y_2,\ldots,y_d \). The morphisms \( f \) and \( g \) provide isomorphism \( f^\flat\map{\Mm_{x}}{\Mm_{f(x)}} \) and \( g_i^\flat\map{K(Y)}{K(Y')} \), which restrict to isomorphisms \( K(X) \longrightarrow K(X') \) and \( K(Y)_i \longrightarrow K(Y')_i \) (and thus also between \( L \) and \( L' \)). Importantly these isomorphisms send primitive elements to primitive elements and thus after identifying the groups \( S_{\pi^{-1}(x)} \) and \( S_{\pi'^{-1}(f(x))} \) using \( g \), the Galois groups are equal.
\end{Remark}

We can always find a dense open subset \( U \subseteq Y \) over which \( \pi \) is unramified. Restricting we obtain a topological covering map \( \pi|_{\pi^{-1}(U)} \). If \( x \in U \) we can consider the \emph{monodromy group} \( M_U(\pi;x) \subseteq S_{\pi^{-1}(x)} \). The following theorem relates the Galois group to the monodromy group.

\begin{Proposition}[{\cite[Section~I.2]{Harris:1979to}}]
  \label{prp:monodromy-equals-galois}
For any \( U \) as above the monodromy group equals the Galois group, \( M_U(\pi;x) = \Gal(\pi;x) \). In particular the monodromy group does not depend on the open neighbourhood chosen to define it.
\end{Proposition}

Remark~\ref{rem:galois-birational-invariant} and Proposition~\ref{prp:monodromy-equals-galois} are the key tools we need to calculate (part of) the Galois group of the spectrum of the Bethe algebras.

\subsection{Proof of Theorem~\ref{thm:main-thm}}
\label{sec:proof-of-main-theorem}

The four varieties we have been investigating and their relationship is summarised by the following diagram.
\begin{equation}
\label{eq:summary-rels}
  \begin{tikzcd}
    \bspec(\lamb)_\mu \arrow{r}{\theta} \arrow{d}{\pi} &
        \Omega(\lamb,\mu^\cc) \arrow[hook]{r}{\iota} \arrow{d}{\rho} &
        \Ss(\lamb,\mu^\cc) \arrow{d}{\eta} &
        \Ss(\lamb,\mu^\cc)(\RR) \arrow[hookleftarrow]{l} \arrow{d}{\eta|_\RR} \\
    \oM{n+1}(\CC) \arrow{r}{=} & \oM{n+1}(\CC) \arrow[hook]{r} & \M{n+1}(\CC) & \M{n+1}(\RR) \arrow[hookleftarrow]{l}
  \end{tikzcd}
\end{equation}

The first claim of Theorem~\ref{thm:main-thm} is that there is a homomorphism \( PJ_n \rightarrow \Gal(\pi;z) \) for some generic point \( z \in \oM{n+1} \). Let \( M_\RR \subseteq S_{\eta^{-1}(z)} \) be the monodromy group of the covering \( \eta|_\RR \). Since \( PJ_n = \pi_1(\M{n+1}(\RR);z) \), by definition we have a surjective homomorphism \( PJ_n \rightarrow M_\RR \). 

Choose a dense open subset \( z \in U \subseteq \M{n+1}(\CC) \) over which \( \eta \) is unramified. We can choose \( U \) so that is contains \( \M{n+1}(\RR) \) by Theorem~\ref{thm:real-points-labelling}. The inclusion of the real points \( \Ss(\lamb,\mu^\cc)(\RR) \) induces an inclusion \( M_\RR \hookrightarrow M_U(\eta;z) \). Proposition~\ref{prp:monodromy-equals-galois} now implies that \( M_U(\eta;z) = \Gal(\eta;z) \). The group \( \Gal(\pi;z) \) is a subgroup of \( S_{\pi^{-1}(z)} \). The morphism \( \iota \circ \theta \) identifies the sets \( \pi^{-1}(z) \) and \( \eta^{-1}(z) \). With this identification fixed, \( \Gal(\pi;z) = \Gal(\eta;z) \) by Remark~\ref{rem:galois-birational-invariant}. Hence we have a homomorphism from \( PJ_n \) onto the subgroup \( M_\RR \subseteq \Gal(\pi;z) \).

The second claim of Theorem~\ref{thm:main-thm} is that for a real  point \( z \in \oM{n+1}(\RR) \) there exists a bijection of sets \( \bspec(\lamb;z)_\mu \rightarrow \Blambmu \) equivariant for the action of \( PJ_n \). The isomorphism \( \theta \) identifies \( \bspec(\lamb;z)_\mu \) with \( \Omega(\lamb,\mu^\cc;z,\infty) \). By definition this identification is equivariant for the action of \( PJ_n \). By Theorem~\ref{thm:labelling-of-fibs} \( \Omega(\lamb,\mu^\cc;z,\infty) \) can be identified with \( \decgd{\lamb,\mu^\cc} \). Now we may use Theorem~\ref{thm:equivariant-map} to find a bijection to \( \Blambmu \) which is equivariant with respect to \( J_n \) (and thus \( PJ_n \)).

\subsection{Proof of Theorem~\ref{thm:equivariant-map} in the fundamental case}
\label{sec:special-case}

We will prove Theorem~\ref{thm:equivariant-map} in the fundamental case for \( \mu \) when \( \lamb = (\sq^n) \). First we recall some facts about the interaction of the RSK-correspondence and the Sch\"utzenberger involution.

\subsubsection{Sch\"utzenberger involution and RSK}
\label{sec:schutz-invol-rsk}

For an integer \( x \in [r] \) let \( x^* = r+1-x \). If \( w = x_1x_2\dots x_n \) is a word in the letters \( 1,2,\ldots,r \) let
\begin{equation*}
  w^* \deq x_n^* x_{n-1}^* \cdots x_1^*.
\end{equation*}
With this notation we have a remarkable duality theorem.

\begin{Theorem}
  \label{thm:duality-theorem}
If \( \RSK(w) = (P,Q) \) then \( \RSK(w^*) = (\xi P, \evac Q) \).
\end{Theorem}

\begin{proof}
See Section~1 of Appendix~A in~\cite{Fulton:1997vaa} for the proof.
\end{proof}

The following proposition gives information about the Q-symbol of subwords. For a skew tableaux \( T \) let \( \Rect(T) \) be the unique tableaux of straight shape, slide equivalent to \( T \).

\begin{Proposition}
  \label{prp:subword-Q-symb}
If \( w = x_1 x_2 \cdots x_n \) is a word with Q-symbol \( Q \) and \( u = x_r x_{r+1} \cdots x_s \) is a (contiguous) subword, then the Q-symbol of \( u \) is \( \Rect(Q|_{r,s}) \).
\end{Proposition}

\begin{proof}
See Proposition~1 in Section~5.1 of~\cite{Fulton:1997vaa}.
\end{proof}

\subsubsection{Standard \( \mu \)-tableaux}
\label{sec:proof-special-case}

Recall from Section~\ref{sec:fundamental-case} the action of \( J_n \) on \( \SYT(\mu) \) by partial Sch\"utzenberger involutions. 

\begin{Proposition}
  \label{prp:std-tab-equiv-bij}
The bijection \( [\crys^{\ox n}]^{\mathsf{sing}}_\mu \longrightarrow \SYT(\mu); w \mapsto \QRSK(w) \), given by taking the Q-symbol of a word, is equivariant for the action of \( J_n \).
\end{Proposition}

\begin{proof}
Recall from Lemma~\ref{lem:generators-of-J_n} that the elements \( s_{1q} \) for \( 1 < q \le n \) generate \( J_n \). Let \( w = b_1b_2\cdots b_n \in [B^{\ox n}]^{\mathsf{sing}}_\mu \) be a highest weight word and let \( Q = \QRSK(w) \). Denote the subword \( b_1\cdots b_q \) by \( w_q \). The word \( s_{1q} \cdot w \) is by definition
\begin{equation*}
  \xi(\xi(b_q) \xi(b_{q-1}) \cdots \xi(b_1))b_{q+1}\cdots b_n = \xi(w_q^*)b_{q+1}\cdots b_n.
\end{equation*}
By Remark~\ref{rem:schutz-on-words} the involution \( \xi \) does not change the  Q-symbol of a word so we have that \( \QRSK(\xi(w_q^*)) = \QRSK(w_q^*) \) and by Theorem~\ref{thm:duality-theorem} \( \QRSK(w_q^*) = \evac\; \QRSK(w_q) \). By considering the definition of the RSK correspondence by the insertion algorithm \( \QRSK(w_q) = Q|_{1,q} \). Thus \( \QRSK(s_{1,q}\cdot w)|_{1,q} = \QRSK(\xi(w_q^*)) = \evac\; Q|_{1,q} \). The remaining letters in the word \( s_{1q}\cdot w \) have not changed, therefore \( \QRSK(s_{1q}\cdot w)|_{q+1,n} = Q|_{q+1,n} \). Thus \( s_{1q}\cdot Q = \QRSK(s_{1q}\cdot w) \).
\end{proof}

\subsection{The proof of Theorem~\ref{thm:equivariant-map}}
\label{sec:general-case}

Recall Proposition~\ref{prp:flipping-decg-is-schutz} gives a \( J_n \)-equivariant bijection \( \SYT(\mu) \longrightarrow \decgd{\sq^n,\mu^\cc} \). Combining this with Proposition~\ref{prp:std-tab-equiv-bij} gives the desired \( J_n \)-equivariant bijection \( [\crys^{\ox n}]^{\mathsf{sing}}_\mu \longrightarrow \decgd{\sq^n,\mu^\cc} \) for the fundamental case (when \( \lamb = (\sq^n) \)). Our strategy for the general case will be the following. We will define embeddings \( [\crys(\lamb)]^{\mathsf{sing}}_\mu \hookrightarrow [B^{\ox n}]^{\mathsf{sing}}_\mu \) and \( \decgd{\lamb,\mu} \hookrightarrow \decgd{\sq^n,\mu} \) that are consistent in the sense that the outer squares of the following diagram commute.

\begin{equation}
  \label{eq:commuting-sqrs}
  \begin{tikzcd}
\left[ \crys(\lamb) \right]^{\mathsf{sing}}_\mu \arrow[hook]{r} \arrow{d}{s_{1q}} & 
    \left[ B^{\ox \tilde{n}} \right]^{\mathsf{sing}}_\mu \arrow[<->]{r} \arrow{d}{\bar{s}_{1q}} & 
    \decgd{\sq^{\tilde{n}},\mu} \arrow{d}{\bar{s}_{1q}} & 
    \decgd{\lamb,\mu} \arrow[hookleftarrow]{l} \arrow{d}{s_{1q}} \\
\left[ \crys(\hat{s}_{1q} \cdot \lamb) \right]^{\mathsf{sing}}_\mu \arrow[hook]{r} & 
    \left[ B^{\ox \tilde{n}} \right]^{\mathsf{sing}}_\mu \arrow[<->]{r} & 
    \decgd{\sq^{\tilde{n}},\mu} & 
    \decgd{\hat{s}_{1q} \cdot \lamb,\mu} \arrow[hookleftarrow]{l}
  \end{tikzcd}
\end{equation}

Here \( s_{1q} \) is a generator of \( J_n \), \( \hat{s}_{1q} \) its image in \( S_n \) and \( \bar{s}_{1q} \in J_{\tilde{n}} \) is a particular element we construct below. Both embeddings will be defined below and are given by a sequence of standard tableaux \( (\Tb) \), where \( T_i \) is a standard \( \lambda_i \)-tableau, we call this object a standard \( \lamb \)-tableau. The inner square of~(\ref{eq:commuting-sqrs}) commutes since this is case \( \lamb = (\sq^{\tilde{n}}) \) which was proven above.

The idea is that the image of \( \bar{s}_{1q} \) in \( S_{\tilde{n}} \) acts by preserving the blocks of the first \( \left| \lambda_1 \right| \) letters, the next \( \left| \lambda_2 \right| \) letters and so on, while permuting these \( n \) blocks in the same way as \( \hat{s}_{1q} \). Let \( m_i = \sum_{j=1}^{i-1} \left| \lambda_j \right| \) and  \( m_i^q = \sum_{j=0}^{i-1} \left| \lambda_{\hat{s}_{1q}(j)} \right| \). Denote the generators of \( J_{\tilde{n}} \) by \( \tilde{s}_{kl} \). Define
\begin{equation*}
  \bar{s}_{1q} = \left( \prod_{i=1}^{q} \tilde{s}_{(m^{q}_i+1)m^{q}_{i+1}} \right) \tilde{s}_{(m_p+1)m_q}.
\end{equation*}

Given a standard \( \lamb \)-tableau, \( \Tb \), we define the embedding \( [\crys(\lamb)]^{\mathsf{sing}}_\mu \hookrightarrow [B^{\ox n}]^{\mathsf{sing}}_\mu \), denoted \( \imath_{\Tb} \) by sending \( b_1\ox \cdots \ox b_n \in [\crys(\lamb)]^{\mathsf{sing}}_\mu \) to the word
\begin{equation*}
   w = x_1x_2\cdots x_{\tilde{n}} \in [B^{\ox n}]^{\mathsf{sing}}_\mu
\end{equation*}
such that \( x_{m_i+1} \cdots x_{m_{i+1}} \) is the unique word with \( P \)-symbol \( b_i \) and Q-symbol \( T_i \). We say that \( w \) having this property has \( \Tb \) as its \emph{\( \lamb \)-partial Q-symbols}.

\begin{Lemma}
  \label{lem:left-sqr-commutes}
The left hand square of~(\ref{eq:commuting-sqrs}),
\begin{equation*}
  \begin{tikzcd}
\left[ \crys(\lamb) \right]^{\mathsf{sing}}_\mu \arrow[hook]{r}{\imath_{\Tb}} \arrow{d}{s_{1q}} & \left[ B^{\ox \tilde{n}} \right]^{\mathsf{sing}}_\mu \arrow{d}{\bar{s}_{1q}} \\
\left[ \crys(\hat{s}_{1q} \cdot \lamb) \right]^{\mathsf{sing}}_\mu \arrow[hook]{r}{\imath_{\hat{s}_{1q} \cdot \Tb}} & \left[ B^{\ox \tilde{n}} \right]^{\mathsf{sing}}_\mu
  \end{tikzcd},
\end{equation*}
commutes.
\end{Lemma}

\begin{proof}
Let \( b = b_1 \ox \cdots \ox b_n \in [\crys(\lamb)]^{\mathsf{sing}}_\mu \). By definition \( \imath_{\hat{s}_{1q} \cdot \Tb} \circ s_{1q} (b) \) has \( \hat{s}_{1q} \cdot \lamb \)-partial Q-symbols \( \hat{s}_{1q} \cdot \Tb \). Our first job is to show the same is true for \( \bar{s}_{1q} \circ \imath_{\Tb}(b) \), i.e.  \( \bar{s}_{1q} \circ \imath_{\Tb}(b) \) lies in the same copy of \( [\crys(\lamb)]^{\mathsf{sing}}_\mu \).

By definition \( w = \imath_{\Tb}(b) \) has \( \lamb \)-partial Q-symbols \( \Tb \), that is, if \( w = x_1 x_2 \cdots x_{\tilde{n}} \) then \( x_{m_i+1}\cdots x_{m_{i+1}} \) has Q-symbol \( T_i \). We will use the notation \( w|_{i,j} \) for the subword \( x_i \cdots x_j \). Let \( \Tb' \) be the  \( \hat{s}_{1q} \cdot \lamb \)-partial Q-symbols of \( \overline{s}_{1q} \cdot w \). If \( i > q \) then \( \left( \bar{s}_{1q} \cdot w \right)|_{m^{s_{1q}}_i+1,m^{s_{1q}}_{i+1}} =  w|_{m_i+1,m_{i+1}} \), so
\begin{equation*}
T_i' = \mathtt{Q} \circ \RSK \left( \left( \bar{s}_{1q} \cdot w \right)|_{m^{s_{1q}}_i+1,m^{s_{1q}}_{i+1}} \right) = \mathtt{Q} \circ \RSK \left( w|_{m_i+1,m_{i+1}} \right) = T_i.
\end{equation*}
Since \( \hat{s}_{1q}(i) = i \), we have \( T_i = T_{\hat{s}_{1q}(i)} \) for \( i > q \). On the other hand, if \( i < q \), let \( s_{1 m_q} \cdot w = y_1y_2 \cdots y_{\tilde{n}} \). By definition,
\begin{align*}
T_{i}' &= \mathtt{Q} \circ \RSK\left( \left( \bar{s}_{1q} \cdot w \right)|_{m^{s_{1q}}_i+1,m^{s_{1q}}_{i+1}} \right) \\
&= \mathtt{Q} \circ \RSK \left( \left( s_{(m^{s_{1q}}_{i}+1)m^{s_{1q}}_{i+1}} \cdot y_1y_2 \cdots y_{\tilde{n}} \right)|_{m^{s_{1q}}_i+1,m^{s_{1q}}_{i+1}} \right), \\ 
\intertext{where we have used the definition of \( \overline{s}_{1q} \). Now applying the definition of the cactus group action on words,}
T_i' &= \mathtt{Q} \circ \RSK \left( \xi (y^*_{m^{s_{1q}}_{i+1}} \cdots y^*_{m^{s_{1q}}_{i}+1}) \right). \\ 
\intertext{Using the fact that \( \xi \) preserves the Q-symbol of a word and applying Theorem~\ref{thm:duality-theorem} we obtain}
T_i' &= \evac \circ \mathtt{Q} \circ \RSK \left( y_{m^{s_{1q}}_{i}+1} \cdots y_{m^{s_{1q}}_{i+1}} \right). \\
\intertext{Now we can apply the rectification property for Q-symbols of subwords from Proposition~\ref{prp:subword-Q-symb}, so} 
T_i' &= \evac \circ \Rect\left( \mathtt{Q} \circ \RSK \left( y_1 \cdots y_{\tilde{n}} \right)|_{m^{s_{1q}}_i+1,m^{s_{1q}}_{i+1}} \right) \\ 
&= \evac \circ \Rect\left( \mathtt{Q} \circ \RSK \left( \xi(x^*_{m_q} \cdots x^*_{1})x_{m_q+1} \cdots x_{\tilde{n}} \right)|_{m^{s_{1q}}_i+1,m^{s_{1q}}_{i+1}} \right) \\ 
&= \evac \circ \Rect\left( \mathtt{Q} \circ \RSK \left( x^*_{m_q} \cdots x^*_{1}x_{m_q+1} \cdots x_{\tilde{n}} \right)|_{m^{s_{1q}}_i+1,m^{s_{1q}}_{i+1}} \right), \\
\intertext{where we have used the definition of \( y_1\ldots y_{\tilde{n}} \) and the fact that \( \xi \) preserves Q-symbols again. Picking out the correct subword and applying Theorem~\ref{thm:duality-theorem} gives} 
T_i' &= \evac \circ \mathtt{Q} \circ \RSK \left( x^*_{m_{s_{1q}(i)+1}} \cdots x^*_{m_{s_{1q}(i)}+1} \right) \\ 
&= \evac \circ \evac \circ \mathtt{Q} \circ \RSK \left( x_{m_{s_{1q}(i)}+1} \cdots x_{m_{s_{1q}(i)+1}} \right). \\ 
\intertext{Since \( \evac \) is an involution,}
T'_i &= \mathtt{Q} \circ \RSK \left( x_{m_{s_{1q}(i)}+1} \cdots x_{m_{s_{1q}(i)+1}} \right) \\ 
&= T_{\hat{s}_{1q}(i)}.
\end{align*}
Now, by Proposition~\ref{prp:local-dual-equiv} and Theorem~\ref{thm:props-of-dual-equiv}, property~\ref{item:words-de-if-Q}, \( \imath_{\hat{s}_{1q} \cdot \Tb} \circ s_{1q} (b) \) and \( \bar{s}_{1q} \circ \imath_{\Tb}(b) \) are dual equivalent words. Since they are by definition highest weight words they are also slide equivalent and thus by Theorem~\ref{thm:props-of-dual-equiv}, property~\ref{item:slide-de-unique}, must be the same word.
\end{proof}

Now we explain the embeddings in the right hand square of~(\ref{eq:commuting-sqrs}). Again let \( \Tb \) be a standard \( \lamb \)-tableaux. For \( \gamma \in \decgd{\lamb,\mu^\cc} \) we can lift this to \( \decgd{\sq^{\tilde{n}},\mu^\cc} \) by simply choosing a representative for each dual equivalence class along the path from \( (1,1) \) to \( (1,n+2) \). We want to do this in a controlled way. Let \( \alpha_i \) be the dual equivalence class allocated to the edge \( (1,i) - (1,i+1) \). Choose a lift \( S_i \) of \( \alpha_i \) such that \( S_i \) is slide equivalent to \( T_i \). Since the intersection of any slide equivalence class and dual equivalence class is a single tableaux, we have a unique choice for \( S_i \). The map \( \jmath_{\Tb} \) is defined by sending \( \gamma \) to the above described decgd in \( \decgd{\sq^{\tilde{n}},\mu^\cc} \).

\begin{Lemma}
  \label{lem:right-sqr-commutes}
The right hand square of~(\ref{eq:commuting-sqrs}),
\begin{equation*}
  \begin{tikzcd}
\decgd{\lamb,\mu^\cc} \arrow[hook]{r}{\jmath_{\Tb}} \arrow{d}{s_{1q}} & \decgd{\sq^{\tilde{n}},\mu^\cc} \arrow{d}{\bar{s}_{1q}} \\
\decgd{\hat{s}_{1q}\cdot \lamb,\mu^\cc} \arrow[hook]{r}{\jmath_{\hat{s}_{1q} \cdot \Tb}} & \decgd{\sq^{\tilde{n}},\mu^\cc}
  \end{tikzcd},
\end{equation*}
commutes.
\end{Lemma}

\begin{proof}
By Corollary~\ref{cor:monodromy-action}, the action of the cactus group on decgd's is given by the rotation of certain triangles. Let \( \gamma \in \decgd{\lamb,\mu^\cc} \). We will first calculate the tableaux defined by the growth along the path from \( (1,1) \) to \( (1,\tilde{n}+2) \) in \( \jmath_{\bar{s}_{1q}\cdot \Tb} (s_{1q}\cdot \gamma) \). As depicted in Figure~\ref{fig:dec-alpha-beta} (for \( q = 4 \)) let \( \alpha_i \) be the dual equivalence class on the edge connecting \( (1,i) \) and \( (1,i+1) \), for \( 1 \le i \le n \) and \( \beta_i \) the dual equivalence class of of the edge connecting \( (q+1,i) \) and \( (q+1,i+1) \) for \( 1 \le i \le q \). Furthermore let \( U_i \) and \( V_i \) be the unique standard tableaux of dual equivalence classes \( \alpha_i \) and \( \beta_i \) respectively which are slide equivalent to \( T_i \).

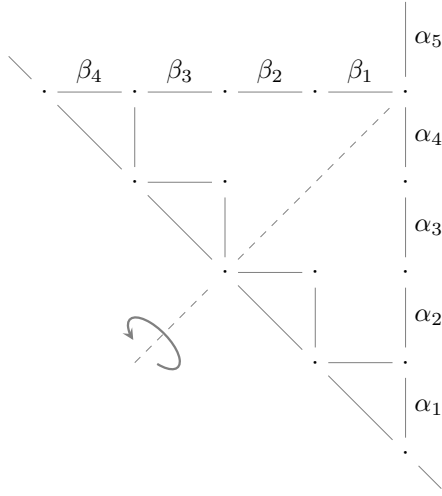
\begin{figure}
  \centering
  \begin{tikzpicture}[scale=0.6]
    \node (00) at (0,0) {\( \cdot \)};
    \node (below) at (1,-1) {\(  \)};
    \node (above) at (-9,9) {\(  \)};
    \node (uq) at (0,8) {\( \cdot \)};
    \node (qq) at (-8,8) {\( \cdot \)};

    \node (mid1) at (-2,2) {\( \cdot \)};
    \node (mid2) at (-4,4) {\( \cdot \)};
    \node (mid3) at (-6,6) {\( \cdot \)};

    \node (tri1) at (0,2) {\( \cdot \)};
    \node (tri2) at (-2,4) {\( \cdot \)};
    \node (tri3) at (-4,6) {\( \cdot \)};
    \node (tri4) at (-6,8) {\( \cdot \)};

    \node (up1) at (-4,8) {\( \cdot \)};
    \node (up2) at (-2,8) {\( \cdot \)};
    \node (side1) at (0,4) {\( \cdot \)};
    \node (side2) at (0,6) {\( \cdot \)};

    \draw[help lines] (below) -- (00) -- (mid1) -- (mid2) -- (mid3) -- (qq) -- (above);
    \draw[help lines] (00) -- node[anchor=west,black] {\( \alpha_1 \)} 
                    (tri1) -- node[anchor=west,black] {\( \alpha_2 \)} 
                   (side1) -- node[anchor=west,black] {\( \alpha_3 \)} 
                   (side2) -- node[anchor=west,black] {\( \alpha_4 \)} 
                      (uq) -- node[anchor=south,black] {\( \beta_1 \)} 
                     (up2) -- node[anchor=south,black] {\( \beta_2 \)} 
                     (up1) -- node[anchor=south,black] {\( \beta_3 \)} 
                    (tri4) -- node[anchor=south,black] {\( \beta_4 \)} 
                    (qq);
    \draw[help lines] (tri1) -- (mid1) -- (tri2) -- (mid2) -- (tri3) -- (mid3) -- (tri4);

    \draw[help lines] (uq) -- (0,10) node[midway,black,anchor=west] {\( \alpha_5 \)};
    \draw[help lines,dashed] (-6,2) --  node[near start,solid] {\AxisRotator[rotate=45]} (mid2) -- (uq);

  \end{tikzpicture}
  \caption{The dual equivalence classes \( \alpha_i \) and \( \beta_i \).}
  \label{fig:dec-alpha-beta}
\end{figure}

The action of \( s_{1q} \) flips the triangle about the axis shown in Figure~\ref{fig:dec-alpha-beta} and preserves the partitions and dual equivalence classes along the path from \( (1,q+1) \) to \( (1,n+2) \) by Proposition~\ref{prp:crossing-walls-in-general}.  By definition, \( \jmath_{\hat{s}_{1q} \cdot \Tb}(s_{1q}\cdot \gamma) \) is then constructed by lifting the appropriate dual equivalence classes to the following tableaux along the path \( (1,1) - (1,\tilde{n}+2) \),
\begin{equation}
\label{eq:seq-of-tableaux-inflipped}
 V_q,  V_{q-1}, \cdots  V_1, U_{q+1}, \cdots U_n.
\end{equation}
This determines \( \jmath_{\hat{s}_{1q} \cdot \Tb}(s_{1q}\cdot \gamma) \). Now we make the same calculation for the other side of the commutative diagram. 

First apply \( \jmath_{\Tb} \) to \( \gamma \), which means lifting the dual equivalence classes along \( (1,1) - (1,n+2) \) (these are the classes \( \alpha_i \)) to \( U_i \). Now apply \( \overline{s}_{1q} \), this means flipping a large triangle and several smaller triangles. Figure~\ref{fig:figure-for-proof} depicts (for \( q=4 \)) the resulting diagram after flipping \emph{only} the large triangle. We have only marked the dual equivalence classes on the vertical and not the actual tableaux.

Now we flip the small triangles, working right to left. The order we flip does not matter as these elements of the cactus group commute. The first triangle is easy, by Proposition~\ref{prp:crossing-walls-in-general} we preserve all the other small triangles as well as the entire path \( (1,\abs{\lambda_{q}}+1) - (1,\tilde{n}+2) \). We end up with \( T_q = V_q \) along the path \( (1,1) - (1,\abs{\lambda_{q}}+1) \).

The triangles further to the right take some more thought, we will try and flip the \( i^{\text{th}} \) triangle. Flipping this triangle preserves all the small triangles to the left and the right as well as everything on the path \( (1,1) - (1,\tilde{n}+2) \) except for the section between \( (1,m^{\hat{s}_{1q}}_i+1) \) and \( (1,m^{\hat{s}_{1q}}_{i+1}+1) \). Locally we have the picture
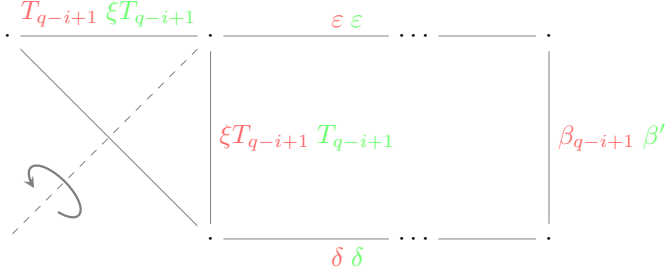
\begin{equation*}
  \begin{tikzpicture}[scale=0.9]
    \node (ii) at (-3,3) {\( \cdot \)};
    \node (i-1) at (0,0) {\( \cdot \)};
    \node (la) at (0,3) {\( \cdot \)};
    \node (0i1) at (5,0) {\( \cdot \)};
    \node (0i) at (5,3) {\( \cdot \)};

    \node (midd) at (3,0) {\( \cdots \)};
    \node (midu) at (3,3) {\( \cdots \)};

    \draw[help lines] (ii) -- node[midway,black,anchor=south] {\( {\color{red!60}T_{q-i+1}}\; {\color{green!60}\xi T_{q-i+1}} \)}
                      (la) -- node[near end,black,anchor=south] {\( {\color{red!60}\varepsilon}\; {\color{green!60}\varepsilon} \)}
                      (midu) -- 
                      (0i) -- node[midway,black,anchor=west] {\( {\color{red!60}\beta_{q-i+1}}\; {\color{green!60}\beta'} \)}
                      (0i1) -- (midd) -- node[near start,black,anchor=north] {\( {\color{red!60}\delta}\; {\color{green!60}\delta} \)}
                      (i-1) -- node[midway,black,anchor=west] {\( {\color{red!60}\xi T_{q-i+1}}\; {\color{green!60}T_{q-i+1}}  \)}
                      (la);

    \draw[help lines] (ii) -- (i-1);
    \draw[help lines,dashed] (la) -- (-3,0) node[near end,solid] {\AxisRotator[rotate=45]};

  \end{tikzpicture}
\end{equation*}
where we have marked the dual equivalence classes and tableaux before the flip in red and after the flip in green. In fact \( \beta' = \beta_{q-i+1} \). To see this denote by \( \eta \) the dual equivalence class of \( T_i \), this is also the dual equivalence class of \( \xi T_i \) by Theorem~\ref{thm:props-of-dual-equiv}~(~\ref{item:normal-all-de}). Thus \( (\delta,\beta_{q-i+1}) \) is the shuffle of \( (\eta,\varepsilon) \). However \( (\delta,\beta') \) is also the shuffle of \( (\eta,\varepsilon) \), thus \( \beta' = \beta_{q-i+1} \).

What we have shown is after flipping all the triangles, (i.e. applying \( \overline{s}_{1q} \) to \( \jmath_{\Tb}(\gamma) \)) the dual equivalence class in the \( i^{\text{th}} \) position on the path \( (1,1) - (1,m_q+1) \) is \( \beta_{q-i+1} \) and the tableaux in this position is thus the unique tableaux in \( \beta_{q-i+1} \) slide equivalent to \( T_{q-i+1} \). By assumption this is \( V_{q-i+1} \). Since we never changed anything on the path \( (1,m_{q}+1) - (1,\tilde{n}+2) \), the sequence of tableaux along the path \( (1,1) - (1,\tilde{n}+2) \) is 
\begin{equation*}
 V_q,  V_{q-1}, \cdots  V_1, U_{q+1}, \cdots U_n,
\end{equation*}
which coincides with~(\ref{eq:seq-of-tableaux-inflipped}). Hence \( \jmath_{\hat{s}_{1q} \cdot \Tb}(s_{1q}\cdot \gamma) = \overline{s}_{1q} \cdot \jmath_{\Tb}(\gamma) \).
\end{proof}

\begin{figure}
  \centering
  \begin{tikzpicture}[scale=0.75]
    \node (00) at (0,0) {\( \cdot \)};
    \node (below) at (1,-1) {\(  \)};
    \node (above) at (-9,9) {\(  \)};
    \node (uq) at (0,8) {\( \cdot \)};
    \node (qq) at (-8,8) {\( \cdot \)};

    \node (mid1) at (-2,2) {\( \cdot \)};
    \node (mid2) at (-4,4) {\( \cdot \)};
    \node (mid3) at (-6,6) {\( \cdot \)};

    \node (tri1) at (0,2) {\( \cdot \)};
    \node (tri2) at (-2,4) {\( \cdot \)};
    \node (tri3) at (-4,6) {\( \cdot \)};
    \node (tri4) at (-6,8) {\( \cdot \)};

    \node (up1) at (-4,8) {\( \cdot \)};
    \node (up2) at (-2,8) {\( \cdot \)};
    \node (side1) at (0,4) {\( \cdot \)};
    \node (side2) at (0,6) {\( \cdot \)};

    \draw[help lines] (below) -- (00) -- (mid1) -- (mid2) -- (mid3) -- (qq) -- (above);
    \draw[help lines] (00) -- node[anchor=west,black] {\( \beta_4, \xi T_4 \)} 
                    (tri1) -- node[anchor=west,black] {\( \beta_3 \)} 
                   (side1) -- node[anchor=west,black] {\( \beta_2 \)} 
                   (side2) -- node[anchor=west,black] {\( \beta_1 \)} 
                      (uq) -- node[anchor=south,black] {} 
                     (up2) -- node[anchor=south,black] {} 
                     (up1) -- node[anchor=south,black] {} 
                    (tri4) -- node[anchor=south,black] {} 
                    (qq);
    \draw[help lines] (tri1) -- node[anchor=south,black] {\( T_4 \)}
                      (mid1) -- node[anchor=west,black] {\( \xi T_3 \)} 
                      (tri2) -- node[anchor=south,black] {\( T_3 \)}
                      (mid2) -- node[anchor=west,black] {\( \xi T_2 \)} 
                      (tri3) -- node[anchor=south,black] {\( T_2 \)}
                      (mid3) -- node[anchor=west,black] {\( \xi T_1 \)} 
                      (tri4);
    \node[anchor=south] at (-7,8) {\( T_1 \)};

    \draw[help lines,dashed] (tri1) -- (-3,-1) node[near end,solid] {\AxisRotator[rotate=45]};
    \draw[help lines,dashed] (tri2) -- (-5,1) node[near end,solid] {\AxisRotator[rotate=45]};
    \draw[help lines,dashed] (tri3) -- (-7,3) node[near end,solid] {\AxisRotator[rotate=45]};
    \draw[help lines,dashed] (tri4) -- (-9,5) node[near end,solid] {\AxisRotator[rotate=45]};

    \draw[help lines,dashed] (mid2) -- (uq) -- (2,10) node[near end,solid] {\AxisRotator[rotate=45]};

  \end{tikzpicture}
  \caption{\( \overline{s}_{1q} \) acting on a decgd}
  \label{fig:figure-for-proof}
\end{figure}

\subsection{Monodromy in Rybnikov's compactification}
\label{sec:rybnikov}

Theorem~\ref{thm:main-thm} is conjectured by Rybnikov~\cite{Rybnikov:2014wh} in a different form. In this section we recall the conjecture and use Theorem~\ref{thm:main-thm} to prove it. 

Rybnikov constructs commutative subalgebras \( \rybalg(\lamb;z)_\mu \) of \( \End(\Llambmu) \) for any point \( z \) in the compactified moduli space \( \M{n+1}(\CC) \). For \( z \in \oM{n+1}(\CC) \), the algebra \( \rybalg(\lamb;z)_\mu \) is simply the Bethe algebra \( \bethe(\lamb;z)_\mu \). Rybnikov shows that for \emph{all real} \( z \in \M{n+1}(\CC) \) the algebra \( \rybalg(\lamb;z)_\mu \) has simple spectrum. This means in particular if we let \( \rybalg(\lamb)_\mu \) be the corresponding family of algebras over \( \M{n+1}(\CC) \) and 
\begin{equation*}
  \rybspec(\lamb)_\mu \deq \Spec \rybalg(\lamb)_\mu,
\end{equation*}
the spectrum of these algebras, then the finite map \( \rybspec(\lamb)_\mu(\RR) \rightarrow \M{n+1}(\RR) \) is a topological covering. The conjecture~\cite[Conjecture~1.6]{Rybnikov:2014wh} which we prove is the following.

\begin{Theorem}
  \label{thm:etingof-conj}
For \( z \in \M{n+1}(\RR) \) the monodromy action of \( PJ_n \) on \( \rybspec(\lamb)_\mu(z) \) is isomorphic to the action of \( PJ_n \) on \( \Blambmu \).
\end{Theorem}

To prove the theorem we require a lemma about the topology of \( \M{n+1}(\RR) \) sitting inside \( \M{n+1}(\CC) \). Let \( U \subset \M{n+1}(\CC) \) be the dense open set over which \( \rybspec(\lamb)_\mu \) is unramified: \( U \) contains \( \M{n+1}(\RR) \). Let \( U_0 = U \cap \oM{n+1}(\CC) \).

\begin{Lemma}
  \label{lem:deform-paths}
Let \( x, y \in \oM{n+1}(\RR) \). Any path in \( \M{n+1}(\RR) \) with endpoints \( x \) and \( y \) is homotopy equivalent to a path in \( U_0 \) with endpoints \( x \) and \( y \).
\end{Lemma}

\begin{proof}
Any path in \( \M{n+1}(\RR) \) is homotopy equivalent to path in \( \M{n+1}(\RR) \) which passes transversally through codimension \( 1 \) cells only. Since \( U_0 \) is open and contains \( \oM{n+1}(\RR) \) it is enough to show we can move our path off such an intersection while remaining in an arbitrarily small neighbourhood of the intersection point.

Locally at the intersection our path is given by \( n \) marked points 
\begin{equation*}
  z_1(t), z_2(t),\ldots,z_n(t) \in \RR
\end{equation*}
depending on a single parameter \( t \), (we fix the last marked point at infinity). Assume for simplicity that \( z_1(t) < z_2(t) < \ldots < z_n(t) \) for \( t < 0 \) and assume the path hits the wall which swaps the order of the marked points \( z_p(t),\ldots,z_q(t) \) at \( t=0 \). We will use \( i \) to denote an integer in the interval \( [p,q] \) and \( j \) to denote an integer between \( 1 \) and \( n \) not in \( [p,q] \). This means \( \lim_{t \rightarrow 0} z_i(t) = y \) for some real number \( y \) which we can assume to be \( 0 \) by using an affine translation if needed. Let \( \epsilon > 0  \) be a small real parameter. Let \( c_i(\epsilon) = \frac{1}{2}(z_i(\epsilon) + z_i(-\epsilon)) \) and \( r_i(\epsilon) = \frac{1}{2}\abs{z_i(\epsilon) - z_i(-\epsilon)} \).  Now define the functions depending on a parameter \( \delta \in [0,1] \)
\begin{equation*}
  f_i(t) =
  \begin{cases}
    \mathrm{sgn}(z_i(\epsilon)) \sqrt{r_i(\epsilon)^2 - (z_i(t) - c_i(\epsilon))^2} & \text{if } t \in (-\epsilon, \epsilon) \\
    0 & \text{otherwise.}
  \end{cases}
\end{equation*}
Consider the path \( z'(t) \) given by \( z'_j(t) = z_j(t) \) and \( z'_i(t) = z_i(t) + f_i(t)\sqrt{-1} \). Let \( U' \) be any neighbourhood of the intersection point. By choosing \( \epsilon \) small enough we can ensure our path is contained in \( U' \). The path \( z'(t) \) depends continuously on \( \epsilon \) and the limit as \( \epsilon \) goes to \( 0 \) is the original path \( z(t) \), hence the paths are homotopy equivalent. We can illustrate this with the picture given in Figure~\ref{fig:homot-path-inters}. We can do this for all transversal intersections and hence the claim is proved.
\end{proof}

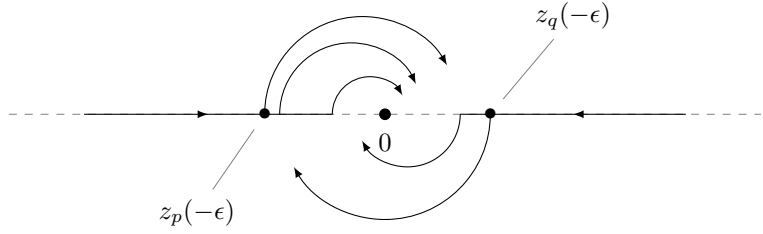
\begin{figure}
  \centering
  \begin{tikzpicture}
    \draw[help lines,dashed] (-5,0) -- (5,0);
    \draw[fill] (0,0) circle (2pt);
    \node[anchor=north,yshift=-4pt] at (0,0) {\( 0 \)};

\begin{scope}[decoration={
    markings,
    mark=at position 0.5 with {\arrow{>}}}
    ] 
    \draw[postaction={decorate}] (-4,0)--(-0.7,0);
    \draw[postaction={decorate}] (4,0)--(1,0);
\end{scope}
    \draw[->] (-0.7,0) arc (180:30:0.5);
    \draw[->] (-1.4,0) arc (180:25:0.95);
    \draw[->] (-1.6,0) arc (180:30:1.3);

    \draw[->] (1,0) arc (0:-150:0.7);
    \draw[->] (1.4,0) arc (0:-150:1.4);

    \node (a1) at (-1.6,0) {\( \bullet \)};
    \node (b1) at (1.4,0) {\( \bullet \)};
    \node[anchor=north] (a) at (-2.5,-1) {\( z_p(-\epsilon) \)};
    \node[anchor=south] (b) at (2.5,1) {\( z_q(-\epsilon) \)};

    \draw[help lines] (a) -- (a1);
    \draw[help lines] (b) -- (b1);

\end{tikzpicture}
  \caption{Homotopy of path off intersection point.}
  \label{fig:homot-path-inters}
\end{figure}

\begin{proof}[Proof of Theorem~\ref{thm:etingof-conj}]
Fix a basepoint \( z \in \oM{n+1}(\RR) \). We identify the fibre with \( \Blambmu \) using the same process described in Section~\ref{sec:proof-of-main-theorem}. Suppose we have a loop \( \gamma_s \) in \( \M{n+1}(\RR) \) given by the element \( s \in PJ_n \). By Lemma~\ref{lem:deform-paths} \( \gamma_s \) is homotopy equivalent to a loop \( \gamma'_s \) contained entirely in \( U_0 \). By Theorem~\ref{thm:main-thm} the monodromy action of \( s \) on the fibre is the same as the action of \( s \) on \( \Blambmu \).
\end{proof}

\bibliographystyle{abbrv}
\bibliography{bibliography}

\begin{thebibliography}{10}

\bibitem{Aguirre:2011gy}
L.~Aguirre, G.~Felder, and A.~P. Veselov.
\newblock {Gaudin subalgebras and stable rational curves}.
\newblock {\em Compositio Mathematica}, 147(5):1463--1478, 2011.

\bibitem{Ariki:2000uw}
S.~Ariki.
\newblock {Robinson-Schensted correspondence and left cells}.
\newblock In {\em Combinatorial methods in representation theory (Kyoto,
  1998)}, pages 1--20. Kinokuniya, Tokyo, 2000.

\bibitem{Kirillov:1995te}
A.~D. Berenstein and A.~N. Kirillov.
\newblock {Groups generated by involutions, Gel fand-Tsetlin patterns, and
  combinatorics of Young tableaux}.
\newblock {\em Rossi skaya Akademiya Nauk. Algebra i Analiz}, 7(1):92--152,
  1995.

\bibitem{Bonnafe:2013ug}
C.~Bonnaf{\'e} and R.~Rouquier.
\newblock {Cellules de Calogero-Moser}.
\newblock {\em arXiv.org}, Feb. 2013.

\bibitem{Chervov:2010fz}
A.~Chervov, G.~Falqui, and L.~Rybnikov.
\newblock {Limits of Gaudin algebras, quantization of bending flows,
  Jucys-Murphy elements and Gelfand-Tsetlin bases}.
\newblock {\em Letters in Mathematical Physics. A Journal for the Rapid
  Dissemination of Short Contributions in the Field of Mathematical Physics},
  91(2):129--150, 2010.

\bibitem{Davis:2003hl}
M.~Davis, T.~Januszkiewicz, and R.~Scott.
\newblock {Fundamental groups of blow-ups}.
\newblock {\em Advances in Mathematics}, 177(1):115--179, 2003.

\bibitem{Devadoss:1999cz}
S.~L. Devadoss.
\newblock {Tessellations of moduli spaces and the mosaic operad}.
\newblock In {\em Homotopy invariant algebraic structures (Baltimore, MD,
  1998)}, pages 91--114. Amer. Math. Soc., Providence, RI, Providence, Rhode
  Island, 1999.

\bibitem{Feigin:1994tc}
B.~Feigin, E.~Frenkel, and N.~Reshetikhin.
\newblock {Gaudin model, Bethe ansatz and critical level}.
\newblock {\em Communications in Mathematical Physics}, 166(1):27--62, 1994.

\bibitem{Fulton:1997vaa}
W.~Fulton.
\newblock {\em {Young tableaux}}, volume~35 of {\em London Mathematical Society
  Student Texts}.
\newblock Cambridge University Press, Cambridge, 1997.

\bibitem{Haiman:1992hu}
M.~D. Haiman.
\newblock {Dual equivalence with applications, including a conjecture of
  Proctor}.
\newblock {\em Discrete Mathematics}, 99(1-3):79--113, 1992.

\bibitem{Harris:1979to}
J.~Harris.
\newblock {Galois groups of enumerative problems}.
\newblock {\em Duke Mathematical Journal}, 46(4):685--724, 1979.

\bibitem{Henriques:2006is}
A.~Henriques and J.~Kamnitzer.
\newblock {Crystals and coboundary categories}.
\newblock {\em Duke Mathematical Journal}, 132(2):191--216, 2006.

\bibitem{Hong:2002vh}
J.~Hong and S.-J. Kang.
\newblock {\em {Introduction to quantum groups and crystal bases}}, volume~42
  of {\em Graduate Studies in Mathematics}.
\newblock American Mathematical Society, Providence, RI, 2002.

\bibitem{Kapranov:1993kz}
M.~M. Kapranov.
\newblock {The permutoassociahedron, Mac Lane's coherence theorem and
  asymptotic zones for the KZ equation}.
\newblock {\em Journal of Pure and Applied Algebra}, 85(2):119--142, 1993.

\bibitem{Knudsen:1983ww}
F.~F. Knudsen.
\newblock {The projectivity of the moduli space of stable curves. II. The
  stacks $M_{g,n}$}.
\newblock {\em Mathematica Scandinavica}, 52(2):161--199, 1983.

\bibitem{Mukhin:2006vt}
E.~Mukhin, V.~Tarasov, and A.~N. Varchenko.
\newblock {Bethe eigenvectors of higher transfer matrices}.
\newblock {\em Journal of Statistical Mechanics: Theory and Experiment},
  2006(8):P08002--44, 2006.

\bibitem{Mukhin:2009et}
E.~Mukhin, V.~Tarasov, and A.~N. Varchenko.
\newblock {Schubert calculus and representations of the general linear group}.
\newblock {\em Journal of the American Mathematical Society}, 22(4):909--940,
  2009.

\bibitem{Mukhin:2010ky}
E.~Mukhin, V.~Tarasov, and A.~N. Varchenko.
\newblock {Gaudin Hamiltonians generate the Bethe algebra of a tensor power of
  the vector representation of $\mathfrak{gl}_N$}.
\newblock {\em Rossi skaya Akademiya Nauk. Algebra i Analiz}, 22(3):177--190,
  2010.

\bibitem{Mukhin:2014wc}
E.~Mukhin, V.~Tarasov, and A.~N. Varchenko.
\newblock {Bethe algebra of Gaudin model, Calogero-Moser space, and Cherednik
  algebra}.
\newblock {\em International Mathematics Research Notices. IMRN},
  (5):1174--1204, 2014.

\bibitem{Reshetikhin:1995vs}
N.~Reshetikhin and A.~N. Varchenko.
\newblock {Quasiclassical asymptotics of solutions to the KZ equations}.
\newblock In {\em Geometry, topology, {\&} physics}, pages 293--322. Int.
  Press, Cambridge, MA, 1995.

\bibitem{Rhodes:1966bc}
F.~Rhodes.
\newblock {On the fundamental group of a transformation group}.
\newblock {\em Proceedings of the London Mathematical Society. Third Series},
  16(1):635--650, 1966.

\bibitem{Rybnikov:2014wh}
L.~Rybnikov.
\newblock {Cactus group and monodromy of Bethe vectors}.
\newblock {\em arXiv.org}, Aug. 2014.

\bibitem{Speyer:2014gg}
D.~E. Speyer.
\newblock {Schubert problems with respect to osculating flags of stable
  rational curves}.
\newblock {\em Algebraic Geometry}, 1(1):14--45, 2014.

\bibitem{Stanley:1999eh}
R.~P. Stanley.
\newblock {\em {Enumerative combinatorics. Vol. 2}}, volume~62 of {\em
  Cambridge Studies in Advanced Mathematics}.
\newblock Cambridge University Press, Cambridge, 1999.

\end{thebibliography}

\end{document}